\documentclass[10pt,a4paper]{article}
\usepackage[utf8]{inputenc}
\usepackage[english]{babel}
\usepackage[T1]{fontenc}
\usepackage{amsmath}
\usepackage{amsthm}
\usepackage{dsfont}
\usepackage{amsfonts}
\usepackage{amssymb}
\usepackage{graphicx}
\usepackage{tabularx}
\usepackage{lmodern}
\usepackage[colorlinks=true,linkcolor=blue]{hyperref}
\usepackage{relsize}
\usepackage{geometry}
\usepackage{mathtools}
\usepackage{enumitem}

\newcommand*{\PP}{\mathbb{P}} 
\newcommand*{\E}{\mathbb{E}} 
\newcommand*{\V}{\mathrm{Var}} 
\newcommand*{\R}{\mathbb{R}} 
\newcommand*{\N}{\mathbb{N}} 
\newcommand*{\Z}{\mathbb{Z}} 
\newcommand*{\C}{\mathbb{C}} 

\newtheoremstyle{theorema}%
    {4pt}
    {4pt}
    {\slshape}
    {}
    {\bfseries}
    {.}
    {.5em}
    {}
\theoremstyle{theorema} 
\newtheorem{theoreme}{Theorem}

\newtheorem{prop}[theoreme]{Proposition}
\newtheorem{lem}[theoreme]{Lemma}

\newtheorem{corollaire}[theoreme]{Corollary}

\theoremstyle{definition} 

\theoremstyle{remark} 
\newtheorem*{remarque}{Remark}

\title{On a limiting point process related to modified permutation matrices}

\author{Valentin B\textsc{ahier} \thanks{valentin.bahier@math.univ-toulouse.fr, Institut de Math\'ematiques de Toulouse, 118 route de Narbonne, F-31062 Toulouse Cedex 9, France.}}

\date{\today}

\begin{document}
\maketitle

\begin{abstract}
We consider random permutation matrices following a one-parameter family of deformations of the uniform distribution, called Ewens' measures, and modifications of these matrices where the entries equal to one are replaced by i.i.d uniform random variables on the unit circle. For each of these two ensembles of matrices, rescaling properly the eigenangles provides a limiting point process as the size of the matrices goes to infinity. If $J$ is an interval of $\R$, we show that, as the length of $J$ tends to infinity, the number of points lying in $J$ of the limiting point process related to modified permutation matrices is asymptotically normal. Moreover, for permutation matrices without modification, if $a$ and $a+b$ denote the endpoints of $J$, we still have an asymptotic normality for the number of points lying in $J$, in the two following cases: [$a$ fixed and $b \to \infty$] and [$a,b \to \infty$ with $b$ proportional to $a$].
\end{abstract}

\section{Introduction}

\subsection{Spectrum of random permutation matrices}


Looking at the counting function of eigenvalues of a random permutation matrix, Wieand~\cite{wieand2000eigenvalue} establishes that the fluctuation of the number of eigenvalues on a fixed arc of the unit circle is asymptotically Gaussian when the size of the matrix goes to infinity, and gives asymptotic expressions of the expectation and variance. 

In \cite{wieand2003permutation}, Wieand tackles more general ensembles of matrices involving random permutations, and shows that her normality result on the fluctuation of the number of eigenvalues holds for these models, with similar behaviors of the expectations and variances.

In these results, Wieand considers uniformly distributed permutations on the symmetric group $\mathfrak{S}_n$. Other measures can be relevant to work with. For instance, the family of Ewens measures are of great interest in population genetics, and have many nice properties which make the study of random permutations simple (some of these properties will be highlighted in the present paper). Arratia, Barbour and Tavaré \cite{arratia2003logarithmic} give and show many results on Ewens measures. Formally, these measures can be defined in the following way:

Let $\theta >0$ and $n\geq 1$. A random permutation $\sigma_n$ of $\mathfrak{S}_n$ follows the Ewens$(\theta )$ distribution if 
\[\forall \pi \in \mathfrak{S}_n, \ \mathbb{P} (\sigma_n = \pi ) = \mathbb{P}_\theta^{(n)} (\pi ) = \frac{\theta^{K(\pi )}}{\theta (\theta + 1) \cdots (\theta + n-1 )}\]
where $K(\pi )$ denotes the total number of cycles of $\pi $ once decomposed into disjoint cycles. The case $\theta =1$ corresponds to the uniform measure.

In this paper we deal with this family of measures on the sets of permutation matrices (we identify the set of the $n$-by-$n$ permutation matrices with $\mathfrak{S}_n$). We also consider modifications of these matrices, where the entries equal to one are replaced by complex numbers of modulus one. These modified permutation matrices can be seen as elements of the wreath product $S^1 \wr \mathfrak{S}_n$, and for the non-zeros entries we take i.i.d random variables uniformly distributed on the unit circle. One main motivation of taking such a law is to bring closer the analogy with the Circular Unitary Ensemble (the $n$-by-$n$ modified permutation matrices form an infinite subgroup of the set of $n$-by-$n$ unitary matrices).



A remarkable property that we would like to point out in this work is the invariance of the behavior of the counting function of eigenvalues by change of scale. Indeed, we observe that the leading coefficients in the asymptotic variances are typically the same through the two following approaches:
\begin{itemize}
	\item Count the eigenvalues in macroscopic or mesoscopic arcs of the unit circle and then let the size of the matrix go to infinity.
	\item Start from the limiting point process of the microscopic landscape of eigenangles, then count the points in any interval and let the length of this interval tend to infinity. 
\end{itemize}

In order to precise this phenomenon, let us recall a few results which will be helpful for comparison purposes. We use the following notations:

Let $(M_n)_{n\geq 1}$ be a sequence of random permutation matrices following the Ewens$(\theta )$ distribution, and let $(\widetilde{M}_n)_{n\geq 1}$ be the sequence of matrices $M_n$ where the entries equal to one are replaced by i.i.d random variables uniformly distributed on the unit circle. \\
For $n\geq 1$, define $X_n^I$ and $\widetilde{X}_n^I$ as the respective numbers of eigenvalues of $M_n$ and $\widetilde{M}_n$ which lie in the arc $I:=\left(\mathrm{e}^{2i\pi \alpha}, \mathrm{e}^{2i\pi \beta}\right]$ of the unit circle, for some $\alpha, \beta$ such that $0\leq \alpha < 1 $ and $\alpha < \beta \leq \alpha + 1$.

For all real numbers $\alpha$ and $\beta$, we set:
\begin{equation}
c_1 = c_1 (\alpha ,\beta ) = \lim_{n\to \infty} \frac{1}{n} \sum_{j=1}^n (\{j\beta \}- \{j\alpha \}),
\end{equation}
\begin{equation}
c_2 = c_2 (\alpha ,\beta ) = \lim_{n\to \infty} \frac{1}{n} \sum_{j=1}^n (\{j\beta \}- \{j\alpha \})^2,
\end{equation}
\begin{equation}
\ell = \ell (\beta - \alpha) = \lim_{n\to \infty} \frac{1}{n} \sum_{j=1}^n \{j(\beta -\alpha)\}(1-\{j(\beta -\alpha)\}),
\end{equation}
where $\{x\}$ denotes the fractional part of $x$. These limits exist and are finite (see e.g.\cite{Bahier2017} for a proof).

\subsubsection*{Macroscopic scale}

The following result has been first established by Wieand~\cite{wieand2000eigenvalue} \cite{wieand2003permutation} in the particular case $\theta=1$, then by Ben Arous and Dang \cite{ben2015fluctuations} for permutation matrices in the general case $\theta>0$, and can be deduced under stronger assumptions from a result of Dang and Zeindler~\cite{dang2014characteristic} on the logarithm of the characteristic polynomial of permutation matrices. 
\begin{prop}\label{prop:Wieand}
Let $0\leq \alpha < 1$ and $ \alpha < \beta \leq \alpha +1$. As $n \to \infty$,
\begin{align*}
\E (X_n^I ) &= n (\beta - \alpha ) - \theta c_1 \log n + o (\log n) \\
\V (X_n^I ) &= \theta c_2 \log n + o (\log n) 
\end{align*}
and 
\begin{align*}
\E (\widetilde{X}_n^I ) &= n (\beta - \alpha ) \\
\V (\widetilde{X}_n^I ) &= \theta \ell \log n + o (\log n) .
\end{align*}
\end{prop}
See \cite{Bahier2017} for a proof of the two last asymptotic equalities.

\subsubsection*{Mesoscopic scale}
In \cite{Bahier2017}, the author of the present paper establishes the following result:

\begin{prop}\label{prop:Valentin}
Assume $I$ to be depending on $n$, of the form $I=I_n:=  \left(\mathrm{e}^{2i\pi \alpha} , \mathrm{e}^{2i\pi (\alpha + \delta_n)} \right]$, where $\alpha \in [0,1)$ and $(\delta_n)$ is a sequence of positive real numbers satisfying \[\left\{\begin{array}{l} 
\delta_n \underset{n\to \infty}{\longrightarrow} 0 \\
n \delta_n \underset{n\to \infty}{\longrightarrow} + \infty .
\end{array} \right.\]
Then, as $n \to +\infty$,
\begin{align*}
\E (X_n^I ) &= n \delta_n - \theta c_1 \log (n\delta_n) + o (\log (n\delta_n)) \\
\V (X_n^I ) &= \theta c_2 \log (n\delta_n) + o (\log (n\delta_n)) 
\end{align*}
and 
\begin{align*}
\E (\widetilde{X}_n^I ) &= n \delta_n \\
\V (\widetilde{X}_n^I ) &= \theta \ell \log (n \delta_n ) + o (\log (n \delta_n )),
\end{align*}
with, denoting by $\kappa$ any arbitrary irrational number, $c_1 = c_1 (\alpha, \kappa)$, $c_2 = c_2 (\alpha , \kappa )$, and $\ell = \ell (\kappa ) = \frac{1}{6}$.
\end{prop}
In fact, the asymptotic of $\E (X_n^I )$ is not computed in \cite{Bahier2017} but can be deduced by the same method as for $\V (X_n^I )$.

Moreover, in both macroscopic and mesoscopic scales, the fluctuations of $X_n^I$ and $\widetilde{X}_n^I$ are asymptotically Gaussian (see \cite{Bahier2017}). 

In this paper we focus on the microscopic landspace of eigenvalues.

\subsubsection*{Microscopic scale}

A \textbf{virtual permutation} is defined as a sequence of permutation $\sigma = (\sigma_n)_{n\geq 1}$ where for all $n$, $\sigma_n \in \mathfrak{S}_n$ and $\sigma_n$ can be obtained from $\sigma_{n+1}$ by simply removing the element $n+1$ in the cycle-decomposition of $\sigma_{n+1}$. A remarkable property of the Ewens measures is that if $\sigma_{n+1}$ follows the Ewens$(\theta )$  distribution on $\mathfrak{S}_{n+1}$, then $\sigma_n$ follows the Ewens$(\theta )$  distribution on $\mathfrak{S}_n$, for every $\theta >0$. Consequently the Ewens measures naturally extend to the space of virtual permutations $\mathfrak{S}$. 

Let $\theta >0 $ and let $\sigma=(\sigma_n)_{n\geq 1}$ be a random virtual permutation following the Ewens$(\theta )$ distribution. 
For $n\geq 1$, let $\ell_{n,j}$ be the length of the $j$-th cycle of $\sigma_n$ in order of appearance (that is to say, in the increasing order of their smallest elements). We complete the sequence $(\ell_{n,j})_{j\geq 1}$ by zeros.
A result of Tsilevitch in \cite{tsilevich1997distribution} states that for all $j\geq 1$, as $n\to \infty$, 
\begin{equation}\label{eq:yj}
y_j^{(n)} := \frac{\ell_{n,j}}{n} \overset{a.s.}{\longrightarrow} y_j,
\end{equation}
where $(y_1, y_2 ,\dots )$ is a random vector following the GEM$(\theta )$ distribution. The rearrangement in decreasing order of the coordinates of a GEM$(\theta )$ vector follows the Poisson-Dirichlet distribution of parameter $\theta$ (PD$(\theta)$), and conversely a size-biased permutation of a PD$(\theta)$ vector has GEM$(\theta )$ distribution. \\
For all $j\geq 1$, $y_j$ has the same law as a product of independent Beta random variables (in the literature this representation of the GEM($\theta)$ distribution is called \textbf{stick breaking process}, or \textbf{residual allocation model}, see \emph{e.g.} \cite{kerov1997stick} and \cite{patil1977diversity}), and a direct calculation shows that there exist $r\in (0,1)$ depending on $\theta$ and independent on $j$, such that 
\begin{equation}\label{eq:yjrj}
\E (y_j ) \leq r^j.
\end{equation}

Now, a basic property on permutation matrices is that their eigenvalues are fully determined by the cycle-structure of their associated permutation. More precisely, each $j$-cycle of any arbitrary given permutation (once decomposed into disjoint cycles) corresponds to a set of eigenvalues equal to the set of $j$-th roots of unity. This supplies us the equalities in distribution
\begin{equation}
X_n^I = \sum_{j=1}^n \mathds{1}_{\ell_{n,j} >0} \sum_{w^{\ell_{n,j}} = 1} \mathds{1}_{w\in I} 
\end{equation}
and
\begin{equation}
\widetilde{X}_n^I = \sum_{j=1}^n \mathds{1}_{\ell_{n,j} >0} \sum_{w^{\ell_{n,j}} = \mathrm{e}^{2i\pi \Phi_{n,j}}} \mathds{1}_{w\in I}, 
\end{equation}
where the $\Phi_{n,j}$ are i.i.d random variables uniformly distributed on $[0,1)$, independent of the $\ell_{n,j}$.

Following the same approach as Najnudel and Nikeghbali in \cite{najnudel2013distribution}, since all the eigenvalues of (modified) permutation matrices are on the unit circle, it can be more practical to consider the eigenangles. The corresponding random measures $\tau (M_n)$ and $\tau (\widetilde{M}_n )$ can be written as
\begin{equation}
\tau (M_n) = \sum_{j=1}^{\infty} \mathds{1}_{\ell_{n,j} >0} \sum_{x \equiv 0 (\mathrm{mod}. \ 2\pi/\ell_{n,j} )} \delta_x
\end{equation}
and
\begin{equation}
\tau (\widetilde{M}_n ) = \sum_{j=1}^{\infty} \mathds{1}_{\ell_{n,j} >0} \sum_{x \equiv 2\pi \Phi_{n,j} (\mathrm{mod}. \ 2\pi/\ell_{n,j} )} \delta_x.
\end{equation}
In particular, this immediately implies that $\tau (M_n) ([0, 2\pi))= \tau (\widetilde{M}_n ) ([0, 2\pi)) = n$, in other words, the average spacing of two consecutive points of their respective corresponding point processes is equal to $2\pi/n$. Thus, if we want to have a convergence of these measures for $n$ going to infinity, we need to rescale them in order to have a constant average spacing, say, one. That is why we introduce the rescaled measures $\tau_n$ and $\widetilde{\tau}_n$, defined as the respective images of $\tau (M_n)$ and $\tau (\widetilde{M}_n )$ by multiplication by $n/2\pi$. One checks that
\begin{equation}
\tau_n = \sum\limits_{j=1}^{+\infty} \mathds{1}_{y_j^{(n)} >0} \sum\limits_{k\in \Z} \delta_{\frac{k}{y_j^{(n)}}}
\end{equation}
and 
\begin{equation}
\widetilde{\tau}_n = \sum\limits_{j=1}^{+\infty} \mathds{1}_{y_j^{(n)} >0} \sum\limits_{k\in \Z} \delta_{\frac{k+ \Phi_{n,j}}{y_j^{(n)}}}.
\end{equation}
Define also the random measures
\begin{equation}
\tau_\infty := \sum\limits_{j=1}^{+\infty} \sum\limits_{k\in \Z \setminus \{0 \}} \delta_{\frac{k}{y_j}} 
\end{equation}
and 
\begin{equation}
\widetilde{\tau}_\infty := \sum\limits_{j=1}^{+\infty} \sum\limits_{k\in \Z} \delta_{\frac{k + \Phi_j}{y_j}}
\end{equation}
where the $y_j$ are given by \eqref{eq:yj}, and the $\Phi_j$ are i.i.d random variables uniformly distributed on $[0,1)$, independent of the $y_j$.

\begin{prop}[Najnudel and Nikeghbali 2010 \cite{najnudel2013distribution}]
\label{prop:Najnudel}
For all continuous functions $f : \R \to \R$ with compact support included in $(0, +\infty )$, 
\[ <\tau_n , f> \overset{a.s.}{\underset{n\to +\infty}{\longrightarrow}}  <\tau_\infty  , f> \] 
under the coupling of virtual permutations, and 
\[ <\widetilde{\tau}_n , f> \overset{d}{\underset{n\to +\infty}{\longrightarrow}}  <\widetilde{\tau}_\infty  , f> .\] 
\end{prop}

%
%

In \cite{najnudel2013distribution} Najnudel and Nikeghbali tackle more general modifications of permutation matrices where the non-zero entries are $\C$-valued (not necessarily of modulus one, so that the matrices are no longer unitary). For the wreath product $S^1 \wr \mathfrak{S}_n$ they also consider more general distributions on $S^1$ (not necessarily the uniform distribution) and provide analog results on their limiting point processes of eigenvalues. 

In the present paper we will restrain ourselves to the study of the limiting point processes related to $(M_n)_{n\geq 1}$ and $(\widetilde{M}_n)_{n\geq 1}$, though the techniques are expected to extend to other ensembles of matrices involving permutations under Ewens measures.

\subsection{Main results and outline of the paper}

In the next section we establish that Proposition~\ref{prop:Najnudel} also holds for indicator functions of intervals. This gives a natural meaning to the convergence of the counting function of the microscopic eigenangles, to a limiting counting function. More precisely, we have the following result:

\begin{prop}\label{prop:limitcounting}
For all positive real numbers $\alpha$ and $\beta$ such that $\alpha<\beta$,
\[\tau_n ((\alpha , \beta ]) \overset{a.s}{\underset{n\to +\infty}{\longrightarrow}}  \tau_\infty ((\alpha , \beta ]) \]
under the coupling of virtual permutations, and   
\[\widetilde{\tau}_n ((\alpha , \beta ]) \overset{d}{\underset{n\to +\infty}{\longrightarrow}}  \widetilde{\tau}_\infty ((\alpha , \beta ]). \]
\end{prop}

\begin{remarque}
It is easy to notice that the laws of the measures $\widetilde{\tau}_n$ and $\widetilde{\tau}_\infty$ are invariant by translation. Thus the second point of Proposition~\ref{prop:limitcounting} is equivalent to say that for all positive real numbers $A$, $\widetilde{\tau}_n ((0 , A ])\overset{d}{\underset{n\to +\infty}{\longrightarrow}}  \widetilde{\tau}_\infty ((0,A])$. Moreover, the choice of including or excluding the endpoints of the interval $(0,A]$ does not have importance for $\widetilde{\tau}_n$ since for all $x\in \R$, $\widetilde{\tau}_n (x)=0$ almost surely. This is clearly not true for $\tau_n$, but it can be proven that for all fixed $x>0$, $\tau_n (x) \rightarrow 0=\tau_\infty (x) $ almost surely as $n \to \infty$ under the coupling of virtual permutations. Indeed, $\tau_n (x) = \sum_{j \geq 1 : \ell_{n,j}>0} \mathds{1}_{x y_j^{(n)} \in \Z}$, so if $0<x<1$ we have $\tau_n (x)=0$ (since $y_j^{(n)} \in (0,1]$ for all $j$ such that $\ell_{n,j}>0$) and if $x\geq 1$ we have for all $j$,
\[\mathds{1}_{y_j^{(n)}>0 , x y_j^{(n)} \in \Z} \leq \mathds{1}_{x y_j^{(n)} \geq 1} \leq \mathds{1}_{\sup_{n} y_j^{(n)} \geq 1/x} \leq \mathds{1}_{C \rho^j \geq 1/x}\]
(see Lemma~\ref{lem:expdecay} for the last inequality) which is summable, and then by dominated convergence we get $\tau_n (x) \rightarrow 0$ a.s., since $xy_j \not \in \Z$ a.s. and then $\mathds{1}_{xy_j^{(n)} \in \Z} \to 0$ a.s. for each $j\geq 1$. \\
More generally, Proposition~\ref{prop:limitcounting} extends to finite numbers of intervals, which immediately implies that both convergences hold for finite combinations of indicator functions.
\end{remarque}

Now, we present our two main results, involving $\tau_\infty$ and $\widetilde{\tau}_\infty$:

\begin{theoreme}\label{thm:asympttautilde}
Let $A>1$. 
\[\frac{\widetilde{\tau}_\infty ([0,A]) - A}{\sqrt{\frac{\theta}{6} \log A}} \underset{A \to +\infty}{\overset{d}{\longrightarrow}} \mathcal{N}(0,1).\]
\end{theoreme}

\begin{theoreme}\label{thm:asympttau}
Let $a$ and $b$ be two positive real numbers such that $a<b$,
\begin{enumerate}[label=(\roman*)]
	\item As $b\to +\infty$,
	\[\E (\tau_\infty  ((a,a+b])) = b - \frac{\theta}{2} \log b + \mathcal{O}_\theta (1) \]
and 
\[\V (\tau_\infty  ((a,a+b])) =  \frac{\theta}{3} \log b + \mathcal{O}_\theta (\sqrt{\log b}) . \]
Moreover, 
\[\frac{\tau_\infty ((a,a+b]) - \E (\tau_\infty ((a,a+b]))}{ \sqrt{\V(\tau_\infty ((a,a+b]))}} \overset{d}{\longrightarrow} \mathcal{N}(0,1). \]
	\item Let $\nu$ be a real number greater than $1$. As $a\to +\infty$ and $b=(\nu -1) a$,
	\[\E (\tau_\infty  ((a,\nu a ]) = (\nu -1) a + \mathcal{O}_\theta (1) \]
and 
\[\V (\tau_\infty  ((a,\nu a])) = \left\{\begin{array}{ll}
\frac{\theta}{6} \left( 1 - \frac{1}{rs} \right) \log a + \mathcal{O}_\theta (\sqrt{\log a}) &\text{ if } \nu = \frac{r}{s} \text{ with } \mathrm{gcd}(r,s)=1 \\
\frac{\theta}{6} \log a + \mathcal{O}_\theta (\sqrt{\log a}) &\text{ if $\nu$ is irrational.}
\end{array}\right.
 \]
Moreover,
\[\frac{\tau_\infty ((a,\nu a]) - \E (\tau_\infty ((a,\nu a]))}{ \sqrt{\V(\tau_\infty ((a,\nu a]))}} \overset{d}{\longrightarrow} \mathcal{N}(0,1). \] 
\end{enumerate}
\end{theoreme}

\begin{remarque}
Note that Theorem~\ref{thm:asympttau} can be related to Propositions~\ref{prop:Wieand} and \ref{prop:Valentin}. In fact, the coefficients in the asymptotic expressions of the expectation and of the variance behave similarly, in the following sense:
\begin{itemize}
	\item Point $(i)$ is linked to the case of a macroscopic arc of the form $I=\left(\mathrm{e}^{2i\pi \alpha}, \mathrm{e}^{2i\pi \beta}\right]$ with $\alpha =0$ and $\beta$ irrational, and also to the case of a mesoscopic arc with the same $\alpha=0$ and replacing $\beta$ by $\delta_n$ (where $\delta_n$ decreases to $0$ slower than $1/n$ as $n$ goes to $\infty$). Indeed, a direct computation (see \cite{wieand2000eigenvalue}) of $c_1$ and $c_2$ gives $c_1 = \frac{1}{2}$ and $c_2=\frac{1}{3}$ for this particular case. 
	\item Point $(ii)$ is linked to the case of a macroscopic arc of the form $I=\left(\mathrm{e}^{2i\pi \alpha}, \mathrm{e}^{2i\pi \beta}\right]$ with $\alpha $ irrational and $\beta$ irrational, and also to the case of a mesoscopic arc with $\alpha$ irrational and $\beta = \alpha + \delta_n$. Indeed, a direct computation (see \cite{wieand2000eigenvalue} and \cite[Appendix B]{Bahier2017}) of $c_1$ and $c_2$ gives $c_1 = 0$ and 
	\[c_2=\left\{\begin{array}{ll}
	\frac{1}{6}\left(1 - \frac{1}{rs} \right) & \text{if } \beta = \frac{r}{s} \alpha  \text{ with } \mathrm{gcd}(r,s)=1 \text{ and } \frac{r}{s} >1  \\
	\frac{1}{6} & \text{if $\alpha$ and $\beta$ are $\Z$-linearly independent} .
	\end{array}\right. \]
\end{itemize}
\end{remarque}

The empirical measures $\tau_\infty$ and $\widetilde{\tau}_\infty$ are related to each other by the following special link:

\begin{prop}\label{prop:transla}
Let $f \in \mathcal{C} (\R , \C)$ with compact support. Let $A>0$. Then
\[< \tau_\infty \circ T_A , f > \overset{d}{\underset{A\to +\infty}{\longrightarrow}} < \widetilde{\tau}_\infty , f >, \]
where $T_A$ is the shift operator defined by $T_A : x \mapsto x+A$.
\end{prop}

The paper follows a linear structure: In Section~\ref{sec:limitcounting} we motivate the study of the considered limiting objects and give a proof of Proposition~\ref{prop:limitcounting}. In Section~\ref{sec:limitMPM} we prove Theorem~\ref{thm:asympttautilde}. In Section~\ref{sec:analogFeller}, we introduce a main tool that we use in Section~\ref{sec:limitMP} for proving Theorem~\ref{thm:asympttau}. This tool is an analog of the ubiquitous Feller coupling, and has interest beyond our study.
Finally, in Section~\ref{sec:transla} we prove Proposition~\ref{prop:transla}.

\section{Two natural limiting counting functions. Proof of Proposition~\ref{prop:limitcounting}}
\label{sec:limitcounting}

We begin with the following lemma:

\begin{lem}\label{lem:expdecay}
There exist $\rho \in (0,1)$ depending on $\theta$, and a random number $C>0$ such that a.s., for all $j\geq 1$, 
\begin{equation}
s_j:=\sup_{m\geq 1} y_j^{(m)} \leq C \rho^j. 
\end{equation}
\end{lem}
\begin{proof}
First, it can be checked that for all $j$, the sequence $\left( \frac{\ell_{N,j}}{N+\theta} \right)_{N\geq 1}$ is a submartingale with respect to the filtration $(\mathcal{F}_N)$ (see \emph{e.g.} \cite{tsilevich1997distribution} for a proof), where $\mathcal{F}_N$ is the $\sigma$-algebra generated by $(\ell_{p,q}, 1\leq p \leq N, q \leq p)$. Moreover, as this submartingale is positive and bounded in $L^2$ (clear since the terms are bounded by $1$), then it follows from Doob's inequality 
\[\E \left( \left(\sup_{N \geq 1 } \frac{\ell_{N,j}}{N+\theta} \right)^2 \right) \leq 4 \sup_{N \geq 1 } \E  \left( \left( \frac{\ell_{N,j}}{N+\theta} \right)^2 \right) \]   
and then, since $\frac{\ell_{N,j}}{N+\theta}$ is lower than $1$, 
\begin{align*}
\E (s_j^2) &\leq 4 (1+\theta)^2 \sup_{N \geq 1 } \E  \left( \frac{\ell_{N,j}}{N+\theta}  \right)  \\
	&= 4  (1+\theta)^2 \lim_{N \to +\infty } \E  \left( \frac{\ell_{N,j}}{N+\theta}  \right) \\
	&\leq 4 (1+\theta)^2 \lim_{N \to +\infty } \E (y_j^{(N)}) \\
	&= 4 (1+\theta)^2 \E ( \lim_{N \to +\infty } y_j^{(N)} ) \\ 
	&= 4 (1+\theta)^2 \E (y_j) \\
	&\leq 4 (1+\theta)^2 r^j.
\end{align*}
where we use the submartingale property for the first equality, the dominated convergence theorem for the second and third equalities, and \eqref{eq:yjrj} for the last inequality. Moreover, using Cauchy-Schwarz inequality we deduce $\E (s_j) \leq 2 (1+ \theta) r^{j/2}$, and finally $\rho:=\frac{1+r^{1/2}}{2}\in (0,1)$ gives
\[\PP( s_j > \rho^j) \leq \frac{1}{\rho^j} \E (s_j) \leq 2(1+\theta) \left(\frac{r^{1/2}}{\rho}\right)^j\]
which is summable in $j$, therefore Borel-Cantelli lemma applies.
\end{proof}

Let $\alpha$ and $\beta$ two real numbers such that $0\leq \alpha<1$ and $\alpha<\beta \leq \alpha+1$.
For all $n$, the random numbers $X_n^I$ and $\widetilde{X}_n^I$ of eigenvalues of $M_n$ and $\widetilde{M}_n$ lying in the arc $I \left(  \mathrm{e}^{2i\pi \alpha} , \mathrm{e}^{2i\pi\beta} \right]$ are given by the following expressions (see \cite{wieand2000eigenvalue}):
\begin{align*}
X_n^I &= \sum_{j=1}^n \mathds{1}_{\ell_{n,j} >0} (\lfloor \ell_{n,j}  \beta \rfloor -  \lfloor \ell_{n,j}  \alpha \rfloor ) \\
	&= n(\beta - \alpha) - \sum_{j=1}^n \mathds{1}_{\ell_{n,j} >0} (\lbrace \ell_{n,j}  \beta \rbrace -  \lbrace\ell_{n,j}  \alpha \rbrace ) 
\end{align*}
and 
\begin{align*}
\widetilde{X}_n^I &= \sum_{j=1}^n \mathds{1}_{\ell_{n,j} >0} (\lfloor \ell_{n,j}  \beta - \Phi_{n,j} \rfloor -  \lfloor \ell_{n,j}  \alpha - \Phi_{n,j} \rfloor ) \\
	&= n(\beta - \alpha) - \sum_{j=1}^n \mathds{1}_{\ell_{n,j} >0} \left( \lbrace \ell_{n,j}  \beta \rbrace -  \lbrace\ell_{n,j}  \alpha \rbrace   - \mathds{1}_{\Phi_{n,j}\leq \lbrace\ell_{n,j}  \beta \rbrace }  +\mathds{1}_{\Phi_{n,j}\leq \lbrace\ell_{n,j}  \alpha \rbrace }   \right) 
\end{align*}
where $(\Phi_{n,j})_{n,j\geq 1}$ is an array of i.i.d random variables uniformly distributed on $[0,1)$, independent of $(\sigma_n)_{n\geq 1}$. \\
If we replace $\alpha$ and $\beta$ respectively by $\alpha/n$ and $\beta/n$, we get
\[X_n^{\left(  \mathrm{e}^{2i\pi \alpha/n} , \mathrm{e}^{2i\pi\beta/n} \right]} = \beta - \alpha - \sum_{j=1}^n \mathds{1}_{y_j^{(n)} >0} (\lbrace y_j^{(n)}  \beta \rbrace -  \lbrace y_j^{(n)} \alpha \rbrace ) \]
and 
\[ \widetilde{X}_n^{\left(  \mathrm{e}^{2i\pi \alpha/n} , \mathrm{e}^{2i\pi\beta/n} \right]} =  \beta - \alpha - \sum_{j=1}^n \mathds{1}_{y_j^{(n)} >0} \left( \lbrace y_j^{(n)}  \beta \rbrace -  \lbrace y_j^{(n)}  \alpha \rbrace   - \mathds{1}_{\Phi_{n,j}\leq \lbrace y_j^{(n)} \beta \rbrace }  +\mathds{1}_{\Phi_{n,j}\leq \lbrace y_j^{(n)} \alpha \rbrace }  \right). \]
From this it seems reasonable to consider the version $n=\infty$ of these quantities, in order to count the points of the limiting point process obtained as the limit of the sequence of eigenangles multiplied by $n/2\pi$ (microscopic scale). The following proposition gives a meaning to the convergence.

\begin{prop}
We have the following convergences:
\begin{equation}
X_n^{\left(  \mathrm{e}^{2i\pi \alpha/n} , \mathrm{e}^{2i\pi\beta/n} \right]} \overset{\text{a.s.}}{\underset{n\to\infty}{\longrightarrow}} \beta - \alpha - \sum_{j=1}^{+\infty} (\lbrace y_j  \beta \rbrace -  \lbrace y_j \alpha \rbrace )
\end{equation}
under the coupling of virtual permutations, and
\begin{equation}
\widetilde{X}_n^{\left(  \mathrm{e}^{2i\pi \alpha/n} , \mathrm{e}^{2i\pi\beta/n} \right]} \underset{n\to\infty}{\overset{d}{\longrightarrow}} \beta - \alpha - \sum_{j=1}^{+\infty} \left( \lbrace y_j \beta \rbrace -  \lbrace y_j  \alpha \rbrace   - \mathds{1}_{\Phi_j \leq \lbrace y_j \beta \rbrace }  +\mathds{1}_{\Phi_j\leq \lbrace y_j \alpha \rbrace }  \right)
\end{equation}
where the $\Phi_j$ are i.i.d random variables uniformly distributed on $[0,1)$, independent of the $y_j^{(n)}$, the $y_j$ and the $\Phi_{n,j}$.
\end{prop}
\begin{remarque}
Note that this proposition is a reformulation of Proposition~\ref{prop:limitcounting}.
\end{remarque}
\begin{proof}
First, we know that a.s., for all $j$, $y_j>0$, hence $\mathds{1}_{y_j^{(n)}>0} \underset{n\to\infty}{\longrightarrow} 1$. \\
Let $x >0$. We are going to show that a.s., $\sum\limits_{j=1}^{+\infty} \lbrace y_j^{(n)} x \rbrace \underset{n\to\infty}{\longrightarrow} \sum\limits_{j=1}^{+\infty} \lbrace y_j x \rbrace$.
By Lemma~\ref{lem:expdecay}, almost surely there exists $\rho \in (0,1)$ and a random number $C>0$ such that for all $j$ and $n$, $y_j^{(n)}\leq C \rho^j$, then
\[\exists j_0 \in \N^*, \quad \forall j > j_0, \quad \forall n \geq 1, \quad y_j^{(n)} x \leq \frac{1}{2}.\]
Fix $j_0$. Letting $n$ tend to infinity, as $y_j^{(n)} \longrightarrow y_j$ a.s., we have for all $j> j_0$, $y_j x\leq \frac{1}{2}$ and then 
\[\vert  \lbrace y_j^{(n)} x \rbrace -  \lbrace y_j x \rbrace \vert = \vert x(y_j^{(n)} - y_j) \vert \underset{n\to\infty}{\longrightarrow} 0. \] 
Moreover, obviously for all $j$ and $n$, 
\[\lbrace y_j^{(n)} x \rbrace \leq y_j^{(n)} x \leq C x \rho^j \]
which is summable in $j$. Hence, by dominated convergence it follows 
\begin{equation}\label{eq:dominatedyj}
\sum\limits_{j=j_0+1}^{+\infty} \lbrace y_j^{(n)} x \rbrace \underset{n\to\infty}{\longrightarrow} \sum\limits_{j=j_0+1}^{+\infty} \lbrace y_j x \rbrace
\end{equation}
almost surely. \\
For $j\leq j_0$ the idea is to take $n$ large enough such that the only $y_j x$ that could pose a challenge are integers (discontinuities of the fractional part function). Let $\varepsilon >0$. There exists $N\in \N^*$ such that for all $n\geq N$, for all $j\leq j_0$, 
\[\vert  \lbrace y_j^{(n)} x \rbrace -  \lbrace y_j x \rbrace \vert \leq \left\{\begin{array}{ll}
	\frac{\varepsilon}{j_0} & \text{if } y_j x \not\in \N  \\
	1 & \text{if } y_j x \in \N
\end{array}\right. ,\]
and then
\[\sum_{j=1}^{j_0}  \vert \lbrace y_j^{(n)} x \rbrace -  \lbrace y_j x \rbrace \vert \leq \varepsilon + \sum_{j=1}^{j_0} \mathds{1}_{y_j x \in \N}.\]
In addition $\sum_{j\leq j_0} \mathds{1}_{y_j x \in \N}=0$ a.s. since it is a finite sum of indicators of negligible events. From \eqref{eq:dominatedyj} we deduce that a.s., 
\begin{equation*}
\sum\limits_{j=1}^{+\infty} \lbrace y_j^{(n)} x \rbrace \underset{n\to\infty}{\longrightarrow} \sum\limits_{j=1}^{+\infty} \lbrace y_j x \rbrace. 
\end{equation*}

It just remains to prove the convergence in distribution of $Q_n:=\sum\limits_{j=1}^{+\infty} \left(\mathds{1}_{\Phi_{n,j} \leq \lbrace y_j^{(n)} \beta \rbrace }  - \mathds{1}_{\Phi_{n,j} \leq \lbrace y_j^{(n)} \alpha \rbrace } \right)$ to $Q:=\sum\limits_{j=1}^{+\infty} \left(\mathds{1}_{\Phi_{j} \leq \lbrace y_j \beta \rbrace } - \mathds{1}_{\Phi_{j} \leq \lbrace y_j \alpha \rbrace } \right)$. \\
Let $t \in \R$. Denoting $\omega_{j,n}:= \lbrace y_j^{(n)} \beta \rbrace - \lbrace y_j^{(n)} \alpha \rbrace$, we have:
\begin{align*}
\E \left[ \mathrm{e}^{i t Q_n } \vert (y_j^{(m)})_{j,m\geq 1}    \right] 
	&= \prod_{j=1}^{+\infty}\left( \mathrm{e}^{it} \omega_{j,n} \mathds{1}_{\omega_{j,n} >0} + \mathrm{e}^{-it} (-\omega_{j,n}) \mathds{1}_{\omega_{j,n} <0} + 1 \times (1- \vert \omega_{j,n} \vert ) \right) \\
	&= \prod_{j=1}^{+\infty}\left( 1 + (\mathrm{e}^{it} -1)\omega_{j,n} \mathds{1}_{\omega_{j,n} >0} - (\mathrm{e}^{-it} -1)  \omega_{j,n} \mathds{1}_{\omega_{j,n} <0}  \right) .
\end{align*}
Taking the logarithm for $j$ large enough, and noting that a.s. there is no $j$ such that $y_j \alpha$ or $y_j \beta$ is integer ($\alpha , \beta >0 $), the dominated convergence theorem ensures that
\[\E \left[ \mathrm{e}^{i t Q_n } \vert (y_j^{(m)})_{j,m\geq 1}  \right]  \underset{n\to\infty}{\longrightarrow} \E \left[ \mathrm{e}^{i t Q} \vert (y_j^{(m)})_{j,m\geq 1}  \right] \]
for almost every realization of $(y_j)_{j\geq 1}$.
Applying once again the dominated convergence theorem, we get
\begin{align*}
\lim_{n\to\infty} \E \left[ \E \left[ \mathrm{e}^{i t Q_n } \vert (y_j^{(m)})_{j,m\geq 1}  \right] \right] &= \E \left[ \lim_{n\to\infty} \E \left[ \mathrm{e}^{i t Q_n } \vert (y_j^{(m)})_{j,m\geq 1}  \right]  \right] = \E \left[ \mathrm{e}^{i t Q} \right] .
\end{align*}
\end{proof}

\section{Limiting point process related to permutation matrices with modification. Proof of Theorem~\ref{thm:asympttautilde}}
\label{sec:limitMPM}

For $A>0$, define
\[\widetilde{X} (A) = A + \sum_{j=1}^{+\infty} \left( \mathds{1}_{\Phi_j \leq \lbrace A y_j \rbrace } - \lbrace A y_j \rbrace \right) .\]
According to the previous section, this random variable counts the number of points in $[0,A]$ of the limiting point process of normalized eigenangles of $\widetilde{M}_n$ when $n$ goes to infinity, i.e we have $\widetilde{X} (A) = \widetilde{\tau}_\infty ([0,A])$. Then, proving Theorem~\ref{thm:asympttautilde} amounts to show 
\begin{equation}
\frac{\widetilde{X} (A) - A}{\sqrt{\frac{\theta}{6} \log (A)}}\underset{A \to \infty}{\overset{d}{\longrightarrow}} \mathcal{N}(0,1).
\end{equation}

Let $A>1$. We first notice that we can write
\[\widetilde{X} (A) - A = \sum_{j=1}^{+\infty} B(p_j)\]
where the $B(p_j)$ are centred Bernoulli random variables of random parameters $p_j:=\{ Ay_j \}$, which are independent conditionally on the $y_j$. \\
Let $\lambda_0 \in \R$ and denote $\lambda := \frac{\lambda_0}{\sqrt{\frac{\theta}{6} \log (A)}}$.
\begin{align*}
\E \left[ \mathrm{e}^{i \lambda \sum\limits_{j=1}^{+\infty} B(p_j)}\  \vert \ (y_m)_{m\geq 1} \right] &= \prod_{j=1}^{+\infty} \E \left[ \mathrm{e}^{i \lambda B(p_j)}\  \vert \ (y_m)_{m\geq 1} \right] \\
	&=\prod_{j=1}^{+\infty} \left( 1 + p_j \left(\mathrm{e}^{i\lambda (1-p_j)}-1\right) + (1-p_j)\left(\mathrm{e}^{-i\lambda p_j}-1\right) \right) \\
	&\underset{A\to\infty}{=}\prod_{j=1}^{+\infty} \left( 1- \frac{\lambda^2}{2} p_j (1-p_j) (1+ \mathcal{O}(\lambda ))  \right).
\end{align*}
Moreover, since the sequence $(p_j(1-p_j))_{j \geq 1}$ is bounded (uniformly in $A$) and using the fact that for all complex numbers $z$ sufficiently close to zero we have $1+z = \exp (z + \mathcal{O} (z^2) )$, it follows that for all $A$ large enough,
\begin{align*}
\E \left[ \mathrm{e}^{i \lambda \sum\limits_{j=1}^{+\infty} B(p_j)} \right] &= \E \left( \exp\left( - \frac{\lambda^2}{2}(1+ \mathcal{O}(\lambda )) \sum_{j=1}^{+\infty}  p_j (1-p_j) \right)  \right). 
\end{align*}
Thus we want to show that
\[\E \left( \exp\left( - \frac{\lambda^2}{2} (1+ \mathcal{O}(\lambda )) \sum\limits_{j=1}^{+\infty}  p_j (1-p_j) \right) \right) \underset{A\to\infty}{\longrightarrow} \mathrm{e}^{-\frac{\lambda_0^2}{2}}.\]
For this purpose, it suffices to show that the random variable $Z_A:= \frac{1}{\log A} \sum\limits_{j=1}^{+\infty} p_j (1-p_j)$ converges in probability to $\frac{\theta}{6}$ when $A$ goes to $+\infty$. Indeed, if we show this, then $Z_A(1+\mathcal{O}(\lambda))$ will clearly converge in probability to $\frac{\theta}{6}$ and it will just remain to apply the definition of the convergence in distribution of $Z_A(1+\mathcal{O}(\lambda))$ (which is positive for all $A$ large enough) to the bounded continuous function $f: x \mapsto \exp \left( -\frac{3 \lambda_0^2}{\theta} x \right)$ on $[0,+\infty )$. \\
Let $\varepsilon >0$. We cut the sum in $Z_A$ into three parts: $j > (1+\varepsilon)\theta \log A$, $(1-\varepsilon)\theta \log A < j \leq (1+\varepsilon)\theta \log A$ and $j\leq (1-\varepsilon)\theta \log A$. \\
In the first regime, we have, noticing that for all integers $k\geq 1$, $\sum\limits_{j=k+1}^{+\infty} y_j \overset{\text{d}}{=} \prod\limits_{j=1}^k U_j$ where the random variables $U_j$ are independent and follow Beta distribution of parameters $\theta$ and $1$,
\begin{align*}
\PP \left( \sum_{j> (1+\varepsilon)\theta \log A} p_j (1-p_j) \geq 1 \right) &\leq \PP \left( \sum_{j> (1+\varepsilon)\theta \log A} Ay_j \geq 1 \right) \\
	&= \PP \left( \prod_{j \leq (1+\varepsilon)\theta \log A} U_j \geq  \frac{1}{A} \right) \\
	&= \PP \left( \frac{1}{(1+\varepsilon)\theta \log A}\sum_{j \leq (1+\varepsilon)\theta \log A} \log U_j \geq  - \frac{1}{(1+\varepsilon)\theta} \right).
\end{align*}
As $\E (\log U_1) = \int_0^1 \log (x) \theta x^{\theta - 1} \mathrm{d}x = \frac{-1}{\theta}$ and  $\frac{-1}{(1+\varepsilon)\theta} > \frac{-1}{\theta}$, then the weak law of large numbers yields  
\[\PP \left( \sum_{j> (1+\varepsilon)\theta \log A} p_j (1-p_j) \geq 1 \right) \underset{A\to\infty}{\longrightarrow} 0,\]
and then
\begin{equation}\label{eq:tautilde1}
\frac{1}{\log A} \sum_{j>(1+\varepsilon)\theta \log A} p_j (1-p_j) \overset{\PP}{\longrightarrow} 0.
\end{equation}
For the $j$ satisfying $(1-\varepsilon)\theta \log A < j \leq (1+\varepsilon)\theta \log A$, 
\begin{equation}\label{eq:tautilde2}
\frac{1}{\log A}  \sum_{(1-\varepsilon)\theta \log A < j \leq (1+\varepsilon)\theta \log A} p_j (1-p_j) \leq \frac{1}{\log A}  \sum_{(1-\varepsilon)\theta \log A < j \leq (1+\varepsilon)\theta \log A} 1 < 2\theta \varepsilon +\frac{1}{\log A}. 
\end{equation}
Finally, for $j\leq (1-\varepsilon)\theta \log A$, let us show that the sum converges in probability to $\frac{\theta}{6} (1-\varepsilon)$. To this end, it is enough to show that its two first moments respectively converge to $\frac{\theta}{6} (1-\varepsilon)$ and $\left(\frac{\theta}{6} (1-\varepsilon)\right)^2$. \\
Recall that for all $j$, $p_j=\{ Ay_j\}$, so the computation of the moments is not obvious. Note that $p_j (1-p_j) = \frac{1}{6} - B_2 (p_j)$, where $B_2$ is the second Bernoulli polynomial ($B_2 (x)=x^2-x+\frac{1}{6}$), which gives a simple expression of its Fourier series. More precisely, for all $x \in \R$ we have the following expansion in Fourier series: 
\[\{x\}(1-\{x\}) = \frac{1}{6} - \frac{1}{2\pi^2}\sum_{k\neq 0} \frac{\mathrm{e}^{2i\pi k x}}{k^2}.\]
Hence,
\[\frac{1}{\log A} \sum_{j\leq (1-\varepsilon)\theta \log A} p_j (1-p_j) = \frac{\lfloor (1-\varepsilon)\theta \log A \rfloor}{6 \log A} - \frac{1}{2\pi^2 \log A}  \sum_{j\leq (1-\varepsilon)\theta \log A} \sum_{k\neq 0} \frac{\mathrm{e}^{2i\pi k A y_j}}{k^2}.\]
For $k\neq 0$, 
\begin{align*}
\E \left( \mathrm{e}^{2i\pi k A y_j} \right) &= \E \left[ \E \left[ \mathrm{e}^{2i\pi k A U_1 \dots U_{j-1} (1-U_j)} \ \vert \ (U_m)_{m\leq j-1}  \right] \right] \\
	&= \E \left[ \int_0^1   \mathrm{e}^{2i\pi k A U_1 \dots U_{j-1} (1-x)} \theta x^{\theta -1} \mathrm{d}x  \right]
\end{align*}
Let $\alpha \in (1-\varepsilon , 1)$ and $\eta \in (0, 1)$ that we will precise at the end of the proof. \\
We write 
\begin{align*}
\int_0^1   \mathrm{e}^{2i\pi k A U_1 \dots U_{j-1} (1-x)} \theta x^{\theta -1} \mathrm{d}x  &= \int_0^1   \mathrm{e}^{2i\pi k A U_1 \dots U_{j-1} (1-x)} \theta x^{\theta -1} \mathrm{d}x \mathds{1}_{U_1 \dots U_{j-1} \leq A^{-\alpha}} \\
	&\quad + \int_0^\eta  \mathrm{e}^{2i\pi k A U_1 \dots U_{j-1} (1-x)} \theta x^{\theta -1} \mathrm{d}x \mathds{1}_{U_1 \dots U_{j-1} > A^{-\alpha}} \\
	&\quad + \int_\eta^1   \mathrm{e}^{2i\pi k A U_1 \dots U_{j-1} (1-x)} \theta x^{\theta -1} \mathrm{d}x \mathds{1}_{U_1 \dots U_{j-1} > A^{-\alpha}} 
\end{align*} 
For the first term on the right-hand side of the equality, 
\begin{align*}
\E \left[ \left\vert \int_0^1   \mathrm{e}^{2i\pi k A U_1 \dots U_{j-1} (1-x)} \theta x^{\theta -1} \mathrm{d}x \mathds{1}_{U_1 \dots U_{j-1} \leq A^{-\alpha}} \right\vert \right] &\leq \int_0^1 \theta x^{\theta -1} \mathrm{d}x \PP ( U_1 \dots U_{j-1} \leq A^{-\alpha} ) \\
	&\leq \PP \left( \prod_{m\leq (1-\varepsilon) \theta \log A} U_m \leq A^{-\alpha}  \right) \\
	&=\PP \left( \frac{1}{(1-\varepsilon) \theta \log A} \sum_{m\leq (1-\varepsilon) \theta \log A}\log U_m \leq -\frac{\alpha}{(1-\varepsilon)\theta} \right) \\
	&\underset{A\to +\infty}{\longrightarrow} 0
\end{align*}
by the weak law of large numbers, since $\frac{-\alpha}{(1-\varepsilon)\theta} < \E (\log U_1) = \frac{-1}{\theta}$. Note that the convergence is uniform in $j$ and $k$. \\
For the second term,  
\[\left\vert \int_0^\eta  \mathrm{e}^{2i\pi k A U_1 \dots U_{j-1} (1-x)} \theta x^{\theta -1} \mathrm{d}x \mathds{1}_{U_1 \dots U_{j-1} > A^{-\alpha}} \right\vert \leq \int_0^\eta \theta x^{\theta - 1} \mathrm{d}x = \eta^\theta .\]
For the third term, an integration by parts gives
\[\int_\eta^1   \mathrm{e}^{2i\pi k A U_1 \dots U_{j-1} (1-x)} \theta x^{\theta -1} \mathrm{d}x = \left[ -\frac{\mathrm{e}^{2i\pi kAU_1 \dots U_{j-1} (1-x)}}{2i\pi kAU_1 \dots U_{j-1} } \theta x^{\theta - 1} \right]_\eta^1 + \int_\eta^1 \frac{\mathrm{e}^{2i\pi kAU_1 \dots U_{j-1} (1-x)}}{2i\pi kAU_1 \dots U_{j-1} } \theta (\theta -1) x^{\theta -2} \mathrm{d}x, \]
so 
\begin{align*}
\left\vert \int_\eta^1   \mathrm{e}^{2i\pi k A U_1 \dots U_{j-1} (1-x)} \theta x^{\theta -1} \mathrm{d}x \mathds{1}_{U_1 \dots U_{j-1} > A^{-\alpha}} \right\vert &\leq \frac{2\theta}{2\pi \vert k \vert AU_1 \dots U_{j-1}}(1+\eta^{\theta -1}) \mathds{1}_{U_1 \dots U_{j-1} > A^{-\alpha}}  \\
	&\leq \frac{\theta}{\pi} (1+\eta^{\theta -1}) A^{\alpha -1 }. 
\end{align*}
It remains to show that we can chose $\eta$ (depending on $A$) such that $\max (\eta^\theta , \eta^{\theta -1} A^{\alpha -1 })$ converges to $0$ when $A$ goes to infinity. If $\theta \geq 1$ it is clear, for instance we can take $\eta = A^{-1}$. If $\theta<1$, $\eta = A^{\frac{1-\alpha}{2(\theta -1)}}$ works. \\
We deduce 
\[\E \left( \mathrm{e}^{2i\pi k A y_j} \right) \underset{A\to +\infty}{=} o (1)\]
where the $o(1)$ is independent of $k$ and $j$. Consequently, 
\begin{align*}
\E \left[ \frac{1}{\log A} \sum_{j\leq (1-\varepsilon)\theta \log A} p_j (1-p_j) \right] &\underset{A\to +\infty}{=} \left( \frac{\theta}{6} (1-\varepsilon) + o(1) \right) - \left(\frac{1}{2\pi^2 \log A}  \sum_{j\leq (1-\varepsilon)\theta \log A} \sum_{k\neq 0} \frac{1}{k^2} \right) o(1) \\ 
	&= \frac{\theta}{6} (1-\varepsilon) + o(1).
\end{align*}
Now, let us show that the second moment converges to $\left(\frac{\theta}{6} (1-\varepsilon)\right)^2$. We have 
\begin{align*}
\left( \sum_{j\leq (1-\varepsilon)\theta \log A} p_j (1-p_j)  \right)^2 &= \frac{\lfloor (1-\varepsilon)\theta \log A \rfloor^2}{36} - \frac{\lfloor (1-\varepsilon)\theta \log A \rfloor}{6\pi^2}  \sum_{j\leq (1-\varepsilon)\theta \log A} \sum_{k\neq 0} \frac{ \mathrm{e}^{2i\pi k A y_j}}{k^2} \\
	&\quad + \frac{1}{4\pi^4} \sum_{j_1,j_2 \leq (1-\varepsilon)\theta \log A } \sum_{k,l\neq 0} \frac{ \mathrm{e}^{2i\pi A (k y_{j_1} + l y_{j_2})}}{k^2 l^2}.
\end{align*}
Let $j_1,j_2 \geq 1$ and $k,l \neq 0$.
\begin{itemize}
	\item If $j_2 > j_1$, then
	\begin{align*}
	\E \left( \mathrm{e}^{2i\pi A (k y_{j_1} + l y_{j_2})} \right) &= \E \left[ \E \left[ \mathrm{e}^{2i\pi A k U_1 \dots U_{j_1-1} (1-U_{j_1})} \mathrm{e}^{2i\pi A l U_1 \dots U_{j_2-1} (1-U_{j_2})}  \ \vert \ (U_m)_{m\leq j_2 -1}   \right]  \right] \\
	&= \E \left[ \mathrm{e}^{2i\pi A k U_1 \dots U_{j_1-1} (1-U_{j_1})} \int_0^1 \mathrm{e}^{2i\pi A l U_1 \dots U_{j_2-1} (1-x)} \theta x^{\theta -1} \mathrm{d}x  \right]
	\end{align*}
	and $\left \vert \mathrm{e}^{2i\pi A k U_1 \dots U_{j_1-1} (1-U_{j_1})} \int_0^1 \mathrm{e}^{2i\pi A l U_1 \dots U_{j_2-1} (1-x)} \theta x^{\theta -1} \mathrm{d}x  \right\vert = \left \vert \int_0^1 \mathrm{e}^{2i\pi A l U_1 \dots U_{j_2-1} (1-x)} \theta x^{\theta -1} \mathrm{d}x  \right\vert$ so, dividing into three pieces as previously we get $\E \left( \mathrm{e}^{2i\pi A (k y_{j_1} + l y_{j_2})} \right) = o (1)$ where the $o(1)$ is independent of $k,l,j_1$ and $j_2$.
	\item If $j_1=j_2$ and $k+l\neq 0$, then
	\[\E \left( \mathrm{e}^{2i\pi A (k y_{j_1} + l y_{j_2})} \right) = \E \left( \mathrm{e}^{2i\pi A (k + l) y_{j_1}} \right) = o(1)\]
	as above.
	\item If $j_1=j_2$ and $k+l =0$, then 
	\[\E \left( \mathrm{e}^{2i\pi A (k y_{j_1} + l y_{j_2})} \right) = 1.\]
\end{itemize}
Thus, 
\begin{align*}
\E \left( \sum_{j_1,j_2 \leq (1-\varepsilon)\theta \log A } \sum_{k,l\neq 0} \frac{ \mathrm{e}^{2i\pi A (k y_{j_1} + l y_{j_2})}}{k^2 l^2} \right) &= o ((\log A)^2) + \sum_{j_1 \leq (1-\varepsilon)\theta \log A } \sum_{k\neq 0} \frac{1}{k^2 (-k)^2} \\
	&= o ((\log A)^2),
\end{align*}
and it follows 
\[\E \left[ \left( \frac{1}{\log A} \sum_{j\leq (1-\varepsilon)\theta \log A} p_j (1-p_j) \right)^2 \right] \underset{A\to +\infty}{=} \left( \frac{\theta}{6} (1-\varepsilon) \right)^2 + o(1). \]
Consequently, 
\begin{equation}\label{eq:tautilde3}
\frac{1}{\log A} \sum_{j\leq (1-\varepsilon)\theta \log A} p_j (1-p_j) \overset{\PP}{\longrightarrow} \frac{\theta}{6}(1-\varepsilon ). 
\end{equation}
Let us now finish to prove the convergence in probability of $Z_A$ to $\frac{\theta}{6}$. For the sake of simplicity, denote 
\[\left\{\begin{array}{l}
Z_{A,>} := \frac{1}{\log A} \sum\limits_{j> (1+\varepsilon)\theta \log A} p_j (1-p_j) \\ 
Z_{A,\star} := \frac{1}{\log A} \sum\limits_{(1-\varepsilon)\theta \log A < j \leq (1+\varepsilon)\theta \log A} p_j (1-p_j) \\ 
Z_{A,\leq} := \frac{1}{\log A} \sum\limits_{j\leq (1-\varepsilon)\theta \log A} p_j (1-p_j). 
\end{array}\right. \]
Combining \eqref{eq:tautilde1}, \eqref{eq:tautilde2} and \eqref{eq:tautilde3}, we have shown: 
\[\left\{\begin{array}{l}
Z_{A,>}  \overset{\PP}{\longrightarrow} 0 \\ 
Z_{A,\star} \leq 2\varepsilon \theta + \frac{1}{\log A} \\ 
Z_{A,\leq} \overset{\PP}{\longrightarrow} \frac{\theta}{6}(1-\varepsilon ).
\end{array}\right. \]
Let $\eta >0$. We have \\
\[ \PP \left( \left\vert Z_A - \frac{\theta}{6} \right\vert > \eta \right) \leq \PP \left( Z_{A,>} > \frac{\eta}{4} \right) + \PP \left( Z_{A,\star} >  \frac{\eta}{4} \right) + \PP \left( \left\vert Z_{A,\leq} - \frac{\theta}{6} (1-\varepsilon ) \right\vert > \frac{\eta}{4} \right) + \PP \left( \frac{\theta}{6} \varepsilon  > \frac{\eta}{4} \right) \]
with 
\[\PP \left( Z_{A,>} > \frac{\eta}{4} \right) \underset{A \to +\infty}{\longrightarrow} 0,\]
\[\PP \left( Z_{A,\star} >  \frac{\eta}{4} \right) \leq \mathds{1}_{2\varepsilon \theta + \frac{1}{\log A} > \frac{\eta}{4}},\]
and 
\[ \PP \left( \left\vert Z_{A,\leq} - \frac{\theta}{6} (1-\varepsilon ) \right\vert > \frac{\eta}{4} \right) \underset{A \to +\infty}{\longrightarrow} 0 \]
whence taking $\varepsilon$ sufficiently close to $0$ (only depending on $\eta$ and $\theta$, for example $\varepsilon = \frac{\eta}{12 \theta}$ fits well), we get
\[\PP \left( \left\vert Z_A - \frac{\theta}{6} \right\vert > \eta \right)  \underset{A \to +\infty}{\longrightarrow} 0,  \]
and the proof is complete.


\section{Continuous analog of the Feller coupling}
\label{sec:analogFeller}


Let $\mathcal{X}$ be a Poisson process with intensity $\frac{\theta}{x} \mathrm{d}x$ on $(0,\infty )$. \\
In this section we are going to show that one can couple the set of random variables $\{y_k,\ k \geq 1\}$ with a set of independent random variables which has the same distribution as $\mathcal{X}\cap (0,1)$, in such a way that these sets are close to each other in $L^2$, in a sense which is made precise below. 

We choose to label the points of $\mathcal{X}$ in the following way:
\begin{equation}
0< \dots < X_3 < X_2 < X_1 <1 \leq X_0 < X_{-1} < X_{-2} < \dots < \infty.
\end{equation}
For all $k \in \Z$, set $Y_k:=X_{k-1}-X_k$. \\
Denote $\mathcal{V}:= \{ 1-X_1, X_1-X_2 , X_2-X_3, \dots  \}$ and $\mathcal{W}:= \{ Y_k : \ k\in \Z , \ Y_k <1 \}$. 

To begin with, note that we have the equalities in law $\{ y_k, \ k\geq 1 \}  \overset{d}{=} \mathcal{V}$ and $\mathcal{W} \overset{d}{=} \mathcal{X}\cap (0,1)$. Indeed, this is a direct consequence of the two following lemmas:
\begin{lem}\label{lem:equalGEM}
\[(y_1 , y_2 , y_3 , \dots ) \overset{d}{=} (1-X_1, X_1-X_2 , X_2-X_3 , \dots  ). \]
\end{lem}

\begin{lem}[Scale invariant spacing lemma]\label{lem:spacing}
\[\{Y_k ,\ k\in \Z \} \overset{d}{=} \{X_k ,\ k\in \Z \}.\]
\end{lem}

We refer to \cite{arratia1998central} for a proof of Lemma~\ref{lem:equalGEM}, and \cite{arratia2006tale} for a proof of Lemma~\ref{lem:spacing}. As mentioned by Arratia in \cite{arratia1998central}, the scale-invariant Poisson process $\mathcal{X}$ is a continuum analog of the sequence $(\xi_j)_{j\geq 1}$ of independent Bernoulli variables involved in the Feller coupling for generating permutations (see \emph{e.g.} \cite{arratia2003logarithmic} for a description of the Feller coupling and related results). Indeed, for $j\geq 1$, the numbers of $j$-spacings between two consecutive ones in the infinite word $\xi_1 \ \xi_2  \dots $ are independent, and similarly by Lemma~\ref{lem:spacing} the spacings obtained from the process $\mathcal{X}$ also form an independent process (in the sense that the numbers of points on disjoint intervals are independent).

Now, we show that the sets $\mathcal{V}$ and $\mathcal{W}$ are close from each other in the following sense: 
\begin{lem}\label{lem:Feller}
There exists a constant number $C(\theta)$ such that
\[\E ((\# \mathcal{V} \Delta \mathcal{W})^2 ) \leq C (\theta).\]
In particular, for all measurable functions $f:\R \to \R$, 
\[\E \left( \left(\sum_{x\in \mathcal{V}} f(x) - \sum_{x\in \mathcal{W}} f(x)\right)^2 \right) \leq C(\theta ) \Vert f \Vert_\infty^2 .  \]
\end{lem}
\begin{proof}
We write
\begin{align*}
\mathcal{V} &= \{1-X_1\} \cup (\{  Y_k ,\  k \in \Z \} \setminus  \{  Y_k ,\  k \leq 1 \} )\\
	&=  \{1-X_1\} \cup (\mathcal{W} \setminus  \{  Y_k  : \  k \leq 1, \ Y_k <1 \}).	
\end{align*} 
Thus, 
\[\mathcal{V} \Delta \mathcal{W} \subseteq \{1-X_1\} \cup \{  Y_k  : \  k \leq 1, \ Y_k <1 \} \]
and then it suffices to show that the number of points in $\{  Y_k  : \  k \leq 0, \ Y_k <1 \}$ is square-integrable. We write
\[\E ((\# \{ Y_k  : \  k \leq 0, \ Y_k <1  \})^2) = \sum_{k=0,-1,-2,\dots } \PP (Y_k < 1) + 2 \sum_{k=-1,-2,\dots} \sum_{\ell = 0, -1, \dots , k+1} \PP (Y_\ell < 1 ,\ Y_k <1).\]
For all $k\leq 0$ and all $x\in (0,\infty)$,
\begin{align*}
\PP (Y_k \geq x) &= \int_1^{+\infty} \PP (Y_k \geq x \ \vert \ X_k = s ) f_{X_k} (s) \mathrm{d}s \\ 
	&= \int_1^{+\infty } \exp \left( - \int_s^{s+x} \frac{\theta}{t} \mathrm{d}t  \right) f_{X_k} (s) \mathrm{d}s 
\end{align*}
where, using basic properties of Poisson processes, $f_{X_k}$ (the density function of $X_k$) is given by 
\[\forall s\geq 1, \  f_{X_k} (s) = \frac{\Lambda (s)^{-k}}{(-k)!} \Lambda^\prime (s) \mathrm{e}^{-\Lambda (s)}, \]
with $\Lambda (s) := \int_1^s \frac{\theta}{y}\mathrm{d}y$.
Consequently, 
\begin{align}
\label{eq:carddiffsym}
\begin{split}
\sum_{k=0,-1,-2,\dots } \PP (Y_k < 1) &= \int_1^{+\infty} \left( 1 - \exp \left( - \int_s^{s+1} \frac{\theta}{t} \mathrm{d}t \right) \right) \frac{\theta}{s} \mathrm{d}s \\
	&= \int_1^{+\infty} \left( 1- \left(\frac{s}{s+1}\right)^\theta \right) \frac{\theta}{s} \mathrm{d}s <+\infty
\end{split}
\end{align}
since $1-\left(\frac{s}{s+1}\right)^\theta \underset{s\to\infty}{\sim} \frac{\theta}{s}$. 

\begin{remarque}
$\sum_{k\leq 0 } f_{X_k} (s)$ is the density probability function of having a point of the Poisson process at $s$, which directly gives $\theta/s$.
\end{remarque}

Now, for all $k,\ell$ such that $0\geq \ell > k$, denoting by $f_{(X_\ell , X_k)}$ the density function of the couple $(X_\ell , X_k)$,
\[\PP (Y_\ell < 1 ,\ Y_k <1) = \int_{s=1}^{+\infty} \int_{t=s}^{+\infty} \PP(Y_\ell \leq 1 , Y_k \leq 1 \ \vert \ (X_\ell , X_k) =(s,t)) f_{(X_\ell , X_k)} (s,t) \mathrm{d}t \mathrm{d}s,\]
where $\PP(Y_\ell \leq 1 , Y_k \leq 1 \ \vert \ (X_\ell , X_k) =(s,t))$ is equal to the probability that there exist at least one point of the Poisson process in the interval $(s,s+1]$ and at least one point in the interval $(t,t+1]$, that we will denote by $A_1 ((s,s+1],(t,t+1])$. Moreover, the numbers of points of every Poisson process in disjoint intervals are independent. Thus, denoting $A_j (J)$ the probability that there exists at least $j$ points of the Poisson process in the interval $J$,
\begin{align*}
\PP (Y_\ell < 1 ,\ Y_k <1) &= \int_{s=1}^{+\infty} \int_{t=s}^{s+1} A_1 ((s,s+1],(t,t+1]) f_{(X_\ell , X_k)} (s,t) \mathrm{d}t \mathrm{d}s \\
	&\quad + \int_{s=1}^{+\infty} \int_{t=s+1}^{+\infty} A_1 ((s,s+1]) A_1 ((t,t+1]) f_{(X_\ell , X_k)} (s,t) \mathrm{d}t \mathrm{d}s.
\end{align*}
Let us compute an explicit expression for $f_{(X_\ell , X_k)}$. For $x,y>1$, 
\begin{align*}
\PP (X_\ell \leq x ,\ X_k \leq y) &= \int_1^{+\infty} \PP (X_\ell \leq x ,\ X_k \leq y \ \vert \ X_\ell =s) f_{X_\ell} (s) \mathrm{d}s \\
	&= \int_1^x \PP (X_k \leq y  \ \vert \ X_\ell = s) f_{X_\ell} (s) \mathrm{d}s \\
	&= \int_1^x A_{\ell -k} ((s,y]) f_{X_\ell} (s) \mathrm{d}s.
\end{align*}
Thus, for $x<y$,
\begin{align*}
\frac{\partial^2}{\partial x \partial y} (\PP (X_\ell \leq x ,\ X_k \leq y) ) &= \frac{\partial }{\partial y} ( A_{\ell-k} ((x,y]) f_{X_\ell} (x) ) \\
	&= f_{X_\ell} (x) \frac{\partial }{\partial y} \left( 1- \sum_{m=0}^{\ell -k-1} \frac{ \left(\int_x^y \frac{\theta}{t} \mathrm{d}t \right)^m}{m!} \exp \left( - \int_x^y \frac{\theta}{t} \mathrm{d}t  \right)  \right) \\
	&= f_{X_\ell} (x) \frac{\theta}{y}  \frac{ \left(\int_x^y \frac{\theta}{t} \mathrm{d}t \right)^{\ell -k-1}}{(\ell -k-1)!} \exp \left( - \int_x^y \frac{\theta}{t} \mathrm{d}t  \right) \\
	&= \frac{\theta^2}{xy} \exp \left( - \int_1^y \frac{\theta}{t} \mathrm{d}t  \right) \frac{1}{(-k-1)!} \binom{-k-1}{-\ell} \left(\int_1^x \frac{\theta}{t} \mathrm{d}t \right)^{-\ell} \left(\int_x^y \frac{\theta}{t} \mathrm{d}t \right)^{-k -1 - (-\ell )}.
\end{align*}
Hence
\begin{align*}
\sum_{k=-1,-2,\dots} \sum_{\ell = 0, -1, \dots , k+1} f_{(X_\ell , X_k)} (x,y)
	&= \frac{\theta^2}{xy}.
\end{align*} 
\begin{remarque}
This sum is the density probability function of having points of the Poisson process simultaneously at $x$ and $y$, which corresponds to the product of intensities $ \frac{\theta}{x} \times \frac{\theta}{y} $.
\end{remarque}
We deduce
\begin{align*}
&\sum_{k=-1,-2,\dots} \sum_{\ell = 0, -1, \dots , k+1}  \int_{s=1}^{+\infty} \int_{t=s}^{s+1} A_1 ((s,s+1],(t,t+1]) f_{(X_\ell , X_k)} (s,t) \mathrm{d}t \mathrm{d}s  \\
	&\qquad = \int_{s=1}^{+\infty} \int_{t=s}^{s+1}  \frac{\theta^2}{st} \mathrm{d}t \mathrm{d}s \\
	&\qquad \leq \theta^2 \int_{s=1}^{+\infty} \frac{1}{s^2} \mathrm{d}s = \theta^2 <+\infty
\end{align*}
and
\begin{align*}
&\sum_{k=-1,-2,\dots} \sum_{\ell = 0, -1, \dots , k+1} \int_{s=1}^{+\infty} \int_{t=s+1}^{+\infty} A_1 ((s,s+1]) A_1 ((t,t+1]) f_{(X_\ell , X_k)} (s,t) \mathrm{d}t \mathrm{d}s \\
 	&\qquad = \int_{s=1}^{+\infty} \int_{t=s+1}^{+\infty} A_1 ((s,s+1]) A_1 ((t,t+1])  \frac{\theta^2}{st} \mathrm{d}t \mathrm{d}s \\
	&\qquad \leq \left( \int_{s=1}^{+\infty}   A_1 ((s,s+1]) \frac{\theta}{s} \mathrm{d}s \right)^2 <+\infty
\end{align*}
by \eqref{eq:carddiffsym}. Consequently,
\[\sum_{k=-1,-2,\dots} \sum_{\ell = 0, -1, \dots , k+1} \PP (Y_\ell < 1 ,\ Y_k <1)<+\infty.\]
This shows the first part of the lemma. The second part of the lemma immediately derives from the first part and the classical inequalities
\[\left\vert \sum_{x\in \mathcal{V}} f(x) - \sum_{x\in \mathcal{W}} f(x) \right\vert \leq  \sum_{x\in \mathcal{V} \Delta \mathcal{W}} \vert f(x) \vert   \leq \Vert f \Vert_\infty \# \mathcal{V} \Delta \mathcal{W}.  \]
\end{proof}

A key result for proving Theorem~\ref{thm:asympttau} is the following simple version of the Campbell's theorem:
\begin{theoreme}[Campbell]\label{thm:campbell}
Let $N$ be a Poisson process with intensity $\Lambda$ on $\R$. Let $f:\R \to \R$ be a measurable function, and denote $T:=\sum_{x\in N} f(x)$. Assume $\int_\R \min (\vert f(x) \vert , 1 ) \Lambda (\mathrm{d}x) < + \infty$. \\ Then for all real numbers $t$,
\[\E (\mathrm{e}^{it T}) = \exp \left( \int_\R \left( \mathrm{e}^{it f(x)} -1 \right)\Lambda (\mathrm{d}x)  \right).\]
Moreover, 
\[\E (T) = \int_\R f(x) \Lambda (\mathrm{d}x) \]
and 
\[\V (T) = \int_\R f(x)^2 \Lambda (\mathrm{d}x)\]
if these integrals converge. 
\end{theoreme}
We refer to \cite{kingman1993poisson} for a proof of Theorem~\ref{thm:campbell}.

From Campbell's theorem, we deduce the following lemma, which will be useful in the next section. 
\begin{lem}\label{lem:gaussianT}
For all $u\in \R_+^*$, let $f_u$ be function from $\R$ to $\R$, and let $T_u:=\sum\limits_{y\in \mathcal{X} \cap (0,1)} f_u (y)$. We assume that the following conditions are satisfied:
\begin{itemize}
\item There exists $K>0$ such that for all $u$, $\vert f_u \vert \leq K$.
\item For all $u$, $\int\limits_0^1 \vert f_u (x) \vert \frac{\theta}{x}\mathrm{d}x < + \infty$.
\item $\int_0^1 f_u (x)^2 \frac{\theta}{x} \mathrm{d}x \underset{u\to +\infty}{\longrightarrow} + \infty$.
\end{itemize}
Then as $u\to +\infty$,
\[ \frac{T_u - \E (T_u)}{\sqrt{\V (T_u) }} \overset{d}{\longrightarrow} \mathcal{N} (0,1).\]
\end{lem}
\begin{proof}
Denote $N_u:= \frac{T_u - \E (T_u)}{\sqrt{\V (T_u) }}$, and $v_u:= \sqrt{\V (T_u) }$.
By Theorem~\ref{thm:campbell} the Fourier transform of $N_u$ is given by 
\begin{align*}
\forall t \in \R , \ \E \left( \mathrm{e}^{it N_u} \right) &= \exp \left( \int_0^1 \left(\mathrm{e}^{i \frac{t}{v_u} f_n (x)} -1\right) \frac{\theta}{x} \mathrm{d}x \right) \mathrm{e}^{-i \frac{t}{v_u} \E (T_u )} \\
	&= \exp \left( - \frac{t^2}{2} + \int_0^1 \left( \sum_{k=3}^{+\infty} \frac{1}{k!} \left( i \frac{t}{v_u} f_u (x) \right)^k  \right) \frac{\theta}{x} \mathrm{d}x \right)
\end{align*}
with 
\begin{align*}
\left\vert \sum_{k=3}^{+\infty} \frac{1}{k!} \left( i \frac{t}{v_u} f_u (x) \right)^k   \right\vert &\leq \sum_{k=3}^{+\infty} \frac{\vert t \vert^k K^{k-2} f_u (x)^2}{k! v_u^k} \\
	&\leq \frac{\vert t \vert^3 K f_u(x)^2}{v_u^3} \exp \left(\frac{\vert t \vert K}{v_u}\right).
\end{align*}
Thus
\begin{align*}
	\int_0^1 \left\vert \sum_{k=3}^{+\infty} \frac{1}{k!} \left( i \frac{t}{v_u} f_u (x) \right)^k  \right\vert  \frac{\theta}{x} \mathrm{d}x  &\leq \exp \left(\frac{\vert t \vert K}{v_u}\right) \frac{\vert t \vert^3 K}{v_u^3} \int_0^1  f_u (x)^2 \frac{\theta}{x} \mathrm{d}x \\
	&= \exp \left(\frac{\vert t \vert K}{v_u}\right) \frac{\vert t \vert^3 K}{v_u} \\
	&\underset{u\to +\infty}{=} \exp ( o(1) )o(1) = o (1)
\end{align*} 
and finally
\begin{align*}
\E \left( \mathrm{e}^{it N_u} \right) &\underset{u\to +\infty}{=} \exp \left( -\frac{t^2}{2} \right) \exp (o (1)) \\
	&= \exp \left( -\frac{t^2}{2} \right) + o (1).
\end{align*}
\end{proof}

\section{Limiting point process related to permutation matrices. Proof of Theorem~\ref{thm:asympttau}}
\label{sec:limitMP}

Let us introduce the random variable $X(s,t)$ which counts the number of points, between the positive real numbers $s$ and $t$, of the limiting point process related to permutation matrices (without modification), i.e
\[X(s,t) = t-s - \sum_{j=1}^{+\infty } (\{ty_j\} - \{sy_j\}).\]
Here, we choose to generate the ensemble $\{y_j , \ j\geq 1\}$ using the continuous analog of the Feller coupling described above.

Let us begin with three lemmas before stating results about $X(s,t)$. Since their proofs are technical we postpone them in Appendix. 

\begin{lem}\label{lem:calcul}
Let $n$ be a positive integer.
\begin{itemize}
\item \[\int_0^1 \frac{\{nx\}}{x} \mathrm{d}x \underset{n\to \infty}{=} \frac{1}{2} \log n + \mathcal{O}(1).\]
\item \[\int_0^1 \{nx\} \log x \mathrm{d}x \underset{n\to \infty}{=} -\frac{1}{2} + \frac{1}{12n} \log n + \mathcal{O}\left( \frac{1}{n} \right).\]
\end{itemize}
\end{lem}

\begin{lem}\label{lem:calcul2}
Let $\ell \in \N^*$. Then
\begin{align*}
\sum_{k=1}^{n-1} \left( 2\left\{ \ell \frac{k}{n} \right\} -1 \right) \log \frac{k}{n} &\underset{n\to\infty}{=} \left[ \frac{\ell}{2} + 2 \sum_{m=1}^{\ell -1} \frac{m}{\ell} \log \frac{m}{\ell} \right] n - \frac{1}{2} \log n + \mathcal{O}(1) \\
\sum_{k=1}^{n-1} \left( 2\left\{ -\ell \frac{k}{n} \right\} -1 \right) \log \frac{k}{n} &\underset{n\to\infty}{=} -\left[ \frac{\ell}{2} + 2 \sum_{m=1}^{\ell -1} \frac{m}{\ell} \log \frac{m}{\ell}\right] n + \frac{1}{2} \log n + \mathcal{O}(1)
\end{align*}
\end{lem}

\begin{lem}\label{lem:calcul3}
Let $p,q$ be two positive integers. 
Then
\begin{align*}
\int_0^1 \frac{(\{ px \} - \{qx\})^2}{x} \mathrm{d}x &= -2 (p-q) \int_0^1 (\{px \} - \{qx \})\log x \mathrm{d}x -\sum_{k=1}^{p-1} \left( 2\left\{ q\frac{k}{p} \right\} - 1 \right) \log \left( \frac{k}{p} \right) \\
	&\qquad - \sum_{j=1}^{q-1} \left( 2\left\{ p\frac{j}{q} \right\} - 1 \right) \log \left( \frac{j}{q} \right) - 2 \sum_{m=1}^{\gcd (p,q) -1} \log \left(\frac{m}{\gcd (p,q)}  \right). 
\end{align*}
\end{lem}

Let $a,b >0$. Define $f_{a,b} : x \mapsto \{(a+b) x \} - \{ax\}$, and denote $S:= \sum\limits_{y\in \mathcal{V}} f_{a,b} (y) = b- X(a,a+b)$ and $T:= \sum\limits_{y\in \mathcal{W}} f_{a,b} (y)$. 

\subsection{Approximation of $S$ by $T$}

Using Lemma~\ref{lem:Feller},
\begin{equation}\label{eq:approxE}
\vert \E (S)- \E (T) \vert \leq \E \vert S - T \vert \leq \Vert f_{a,b} \Vert_\infty C(\theta ) \leq C ( \theta )
\end{equation}
and
\begin{equation}\label{eq:approxVar}
\vert \sqrt{\V (S)} - \sqrt{ \V (T)} \vert \leq \sqrt{ \V ( S - T ) } \leq \sqrt{  \E ((S-T)^2) } \leq \Vert f_{a,b} \Vert_\infty \sqrt{C(\theta )} \leq \sqrt{C(\theta )}.  
\end{equation}
Therefore, as soon as $\V (T) \to +\infty$ we will get
\[\frac{\E (S) - S}{\sqrt{\V (S)}} - \frac{\E (T) - T}{\sqrt{\V (T)}} \overset{\PP}{\longrightarrow} 0. \]
Moreover, it is easy to check that $\int_0^1 \frac{\vert f_{a,b}  (x)\vert}{x} \mathrm{d}x < +\infty$, as for all $x\in \left( 0 , \frac{1}{a+b} \right)$ we have $f_{a,b} (x) = bx$. We deduce by Lemma~\ref{lem:gaussianT} and Slusky's theorem that as soon as $\V (T) \to +\infty$, we have $\frac{\E (S) - S}{\sqrt{\V (S)}} \overset{d}{\longrightarrow} \mathcal{N}(0,1)$, \emph{i.e}
$\frac{X(a,a+b) - \E (X(a,a+b))}{\sqrt{\V (X(a,a+b))}} \overset{d}{\longrightarrow} \mathcal{N}(0,1)$. \\
Furthermore, Theorem~\ref{thm:campbell} applies and gives 
\[\E (T) =\theta \int_0^1 \frac{f_{a,b}(x)}{x}  \mathrm{d}x \]
and 
\[\V (T) =\theta \int_0^1 \frac{f_{a,b}(x)^2}{x} \mathrm{d}x.\] 

\subsection{Proof of point $(i)$ of Theorem~\ref{thm:asympttau}}

\subsubsection{Proof for $a,b \in \N^*$}

Assume $a,b\in \N^*$. Then Lemmas~\ref{lem:calcul},~\ref{lem:calcul2} and~\ref{lem:calcul3} provides all we need for the computation of the asymptotics of $\E (T)$ and $\V (T)$ when $b$ tends to infinity. \\

Denote $p=a$ and $q=a+b$. 
If $a$ is fixed and $b$ goes to infinity, then using Lemma~\ref{lem:calcul},
\begin{equation}
\E (T) = \frac{\theta}{2} \log b + \mathcal{O}_\theta (1)
\end{equation}
and
\begin{align*}
-2(p-q) \int_0^1 (\{px \} - \{qx \})\log x \mathrm{d}x &= 2(q-p) \left( \int_0^1 \{px \} \log x \mathrm{d}x + \frac{1}{2} - \frac{1}{12} \frac{\log q}{q} + \mathcal{O}\left(\frac{1}{q}\right) \right)  \\
	&=  \frac{q-p}{q} \left(\left( 1 + 2 \int_0^1 \{px \} \log x \mathrm{d}x \right) q - \frac{1}{6} \log q + \mathcal{O}(1) \right) 
\end{align*}
with $\frac{q-p}{q} = 1 + \mathcal{O}(\frac{1}{q})$.
Besides, clearly since $p$ is fixed
\[\sum_{k=1}^{p-1} \left( 2\left\{ k\frac{q}{p} \right\} - 1 \right) \log \left( \frac{k}{p} \right) = \mathcal{O}(1), \]
and using Lemma~\ref{lem:calcul2},
\[\sum_{j=1}^{q-1} \left( 2\left\{ p\frac{j}{q} \right\} - 1 \right) \log \left( \frac{j}{q} \right) = \left[ \frac{p}{2} + 2 \sum_{m=1}^{p-1} \frac{m}{p} \log \frac{m}{p} \right]q - \frac{1}{2} \log q + \mathcal{O}(1). \]
Furthermore, as $\gcd(p,q)\leq p$,
\[\sum_{m=1}^{\gcd (p,q) -1} \log \left(\frac{m}{\gcd (p,q)}  \right) = \mathcal{O}(1).\]
We deduce by Lemma~\ref{lem:calcul3}
\begin{align*}
\int_0^1 \frac{(\{px \} - \{qx\})^2}{x} \mathrm{d}x &= \left( 1 + 2 \int_0^1 \{px \} \log x \mathrm{d}x  - \frac{p}{2} - 2 \sum_{m=1}^{p-1} \frac{m}{p} \log \frac{m}{p} \right) q + \left( - \frac{1}{6} + \frac{1}{2} \right) \log q + \mathcal{O}(1) \\
	&= \frac{1}{3} \log q + \mathcal{O}(1).
\end{align*}
Indeed,
\begin{align*}
\int_0^1 \{px \} \log x \mathrm{d}x &= \sum_{k=0}^{p-1} \int_{\frac{k}{p}}^{\frac{k+1}{p}} (px-k) \log x \mathrm{d}x \\
	&= p \int_0^1 x \log x \mathrm{d}x - \sum_{k=1}^{p-1} \left(\sum_{j=1}^k 1\right)\left(\left( \frac{k+1}{p} \log \frac{k+1}{p} - \frac{k+1}{p} \right) - \left(\frac{k}{p}\log \frac{k}{p} - \frac{k}{p} \right)\right) \\
	&= -\frac{p}{4} + \sum_{j=1}^p \left( \frac{j}{p} \log \frac{j}{p} - \frac{j}{p} \right) +p \\
	&= \frac{p}{4} +\sum_{j=1}^{p-1} \frac{j}{p} \log \frac{j}{p} - \frac{1}{2}.
\end{align*}
It follows 
\begin{equation}
\V (T) = \frac{\theta}{3} \log b + \mathcal{O}_\theta (1), 
\end{equation}
and as $\left\vert \sqrt{\V (T)} -\sqrt{\V (S)} \right\vert = \mathcal{O}_\theta (1)$, then
\begin{align*}
\V (X(a,a+b))= \V (S) &= \left(\left( \frac{\theta}{3} \log b + \mathcal{O}_\theta (1) \right)^{1/2} + \mathcal{O}_\theta (1) \right)^2 \\
	&= \frac{\theta}{3} \log b + \mathcal{O}_\theta ( \sqrt{\log b} ).
\end{align*}

From the previous paragraph we deduce
\begin{equation}
\frac{X(a,a+b) - \E (X(a,a+b))}{\sqrt{\V (X(a,a+b))}} \underset{b \to \infty}{\overset{d}{\longrightarrow}} \mathcal{N}(0,1).
\end{equation}

\subsubsection{Generalization for all $a,b$}

Assume now $a,b$ to be positive real numbers, with $b>1$. We have the inequalities
\begin{align*}
0 \leq X(a,a+b) - X(\lceil a \rceil , \lfloor a \rfloor + \lfloor b \rfloor ) &= X(a, \lceil a \rceil) + X ( \lfloor a \rfloor + \lfloor b \rfloor , a+b) \\
	&\leq X (a , \lceil a \rceil) + X (\lfloor a \rfloor + \lfloor b \rfloor , \lfloor a \rfloor + 2  + \lfloor b \rfloor ),
\end{align*}
with 
\begin{equation}\label{eq:EX}
\E (  X (\lfloor a \rfloor + \lfloor b \rfloor , \lfloor a \rfloor + 2  + \lfloor b \rfloor ) ) \underset{b\to \infty}{=} \mathcal{O}_\theta (1)
\end{equation}
by \eqref{eq:approxE} and Lemma~\ref{lem:calcul}.  
Moreover, 
\begin{align*}
&\left\vert \sqrt{\V (X(a,a+b))} - \sqrt{\V (X(\lceil a \rceil , \lfloor a \rfloor + \lfloor b \rfloor ))} \right\vert \\
	&\quad \leq \sqrt{\V (X(a,a+b) - X(\lceil a \rceil , \lfloor a \rfloor + \lfloor b \rfloor )) } \\
	&\quad = \sqrt{\V (X(a,\lceil a \rceil ) + X ( \lfloor a \rfloor + \lfloor b \rfloor , a+b) } \\
	&\quad \leq \sqrt{3} \sqrt{ \V (X(a,\lceil a \rceil )) + \V (X ( \lfloor a \rfloor + \lfloor b \rfloor , \lfloor a \rfloor +b)) + \V (X ( \lfloor a \rfloor + b , a+b)) } \\
	&\quad \leq 3 + \V (X(a,\lceil a \rceil )) + \V (X ( \lfloor a \rfloor + \lfloor b \rfloor , \lfloor a \rfloor +b)) + \V (X ( \lfloor a \rfloor + b , a+b)).
\end{align*}
Let us show that $\V (X ( \lfloor a \rfloor + \lfloor b \rfloor , \lfloor a \rfloor +b)) + \V (X ( \lfloor a \rfloor + b , a+b)) = \mathcal{O}_\theta (1)$. \\
For the first term, from \eqref{eq:approxVar} it is enough to show 
\begin{equation}
\int_0^1 \frac{(\{ (\lfloor a \rfloor +b)x \} -  \{ (\lfloor a \rfloor + \lfloor b \rfloor)x \})^2   }{x} \mathrm{d}x \underset{b\to \infty}{=} \mathcal{O}(1).
\end{equation}
For the sake of simplicity, denote $m= \lfloor a \rfloor$. We have, for all $x\in [0,1)$, 
\begin{align*}
\{ (m + b)x \} -  \{ (m + \lfloor b \rfloor )x \} &= \{ (m + b)x - (m + \lfloor b \rfloor )x \} - \mathds{1}_{ \{ (m + b)x - (m + \lfloor b \rfloor )x \} + \{ (m + \lfloor b \rfloor )x \} > 1} \\
	&= \{ b \} x - \mathds{1}_{\{ b \} x + \{ (m + \lfloor b \rfloor )x \} > 1} 
\end{align*}
so that 
\[ \int_0^1 \frac{(\{ (\lfloor a \rfloor +b)x \} -  \{ (\lfloor a \rfloor + \lfloor b \rfloor)x \})^2   }{x} \mathrm{d}x \underset{b\to \infty}{=} I_b + \mathcal{O}(1) \]
where $I_b := \int_0^1 \frac{1}{x} \mathds{1}_{\{ b \} x + \{ (m + \lfloor b \rfloor )x \} > 1}  \mathrm{d}x $. We want to show $I_b = \mathcal{O}(1)$. We cut the integral as follows:
\[I_b = \int_0^{1/(m+\lfloor b \rfloor )} \frac{1}{x} \mathds{1}_{\{ b \} x + \{ (m + \lfloor b \rfloor )x \} > 1} \mathrm{d}x  + \sum_{k=1}^{m+\lfloor b \rfloor -1} \int_{k/(m+\lfloor b\rfloor)}^{(k+1)/(m+\lfloor b\rfloor)} \frac{1}{x} \mathds{1}_{\{ b \} x + \{ (m + \lfloor b \rfloor )x \} > 1} \mathrm{d}x.\]
We have 
\begin{align*}
\int_0^{1/(m+\lfloor b \rfloor )} \frac{1}{x} \mathds{1}_{\{ b \} x + \{ (m + \lfloor b \rfloor )x \} > 1} \mathrm{d}x &= \int_0^{1/(m+\lfloor b \rfloor )} \frac{1}{x} \mathds{1}_{ (m + b)x > 1} \mathrm{d}x \\
	&= \int_{1/(m+b)}^{1/(m+\lfloor b \rfloor )} \frac{1}{x} \mathrm{d}x = \log \left( \frac{m + b}{m+ \lfloor b \rfloor} \right) \\
	&\underset{b \to \infty}{\longrightarrow} 0,
\end{align*}
and for all $k \geq 1$, 
\begin{align*}
\int_{k/(m+\lfloor b\rfloor)}^{(k+1)/(m+\lfloor b\rfloor)} \frac{1}{x} \mathds{1}_{\{ b \} x + \{ (m + \lfloor b \rfloor )x \} > 1} \mathrm{d}x &= \int_{k/(m+\lfloor b\rfloor)}^{(k+1)/(m+\lfloor b\rfloor)}   \frac{1}{x} \mathds{1}_{x  > (k+1)/(m+b)} \mathrm{d}x \\
	&\leq  \frac{1}{k/(m+\lfloor b\rfloor )} \int_{(k+1)/(m+b)}^{(k+1)/(m+\lfloor b\rfloor)} \mathrm{d}x \\
	&= \frac{k+1}{k} (m + \lfloor b \rfloor) \left( \frac{1}{m + \lfloor b \rfloor } - \frac{1}{m+b} \right)\\
	&\leq \frac{2}{m + \lfloor b \rfloor}.
\end{align*}
Hence $I_b \leq o(1) + 2 = \mathcal{O}(1)$. 
A very similar computation gives $\V (X ( \lfloor a \rfloor + b , a+b)) = \mathcal{O}_\theta (1)$.\\ 
We deduce
\[ \sqrt{\V (X(a,a+b))} = \sqrt{\V (X(\lceil a \rceil , \lfloor a \rfloor + \lfloor b \rfloor ))} + \mathcal{O}_\theta (1) \]
which, combining with \eqref{eq:approxVar}, yields
\begin{equation}\label{eq:varX}
\sqrt{\V (X(a,a+b))} = \sqrt{ \frac{\theta}{3} \log b } + \mathcal{O}_\theta (1).
\end{equation}
Using Markov inequality, \eqref{eq:EX} and \eqref{eq:varX} imply 
\[ \frac{X(a,a+b) -  X(\lceil a \rceil , \lfloor a \rfloor + \lfloor b \rfloor )}{ \sqrt{\V (X(a,a+b))} }  \overset{\PP}{\underset{b \to \infty}{\longrightarrow}} 0 \]
and applying Slutsky's lemma it follows
\[\frac{X(a,a+b) - \E (X(a,a+b))}{\sqrt{\V (X(a,a+b))}} \underset{b \to \infty}{\overset{d}{\longrightarrow}} \mathcal{N} \left( 0, 1\right)\]
which completes the proof.

\subsection{Proof of point $(ii)$ of Theorem~\ref{thm:asympttau}}

\subsubsection{Case $\nu$ rational}

\begin{lem}\label{lem:transfo}
Let $f$ be a non-negative function on $[0,1]$ such that $f$ is integrable on $[0,1]$ and $x\mapsto \frac{f(x)}{x}$ is integrable in the neighbourhood of $0$. Let $t\in \R$. Then
\[\int_0^1 \frac{f(\{tx\})}{x} \mathrm{d}x \underset{t\to +\infty}{=} (\log t) \int_0^1 f(x) \mathrm{d}x + \mathcal{O}(1).\] 
\end{lem}
\begin{proof}
Let $t\geq 2$. It suffices to write
\[\int_0^1 \frac{f(\{ tx\})}{x} \mathrm{d}x = \int_0^t \frac{f(\{ x\})}{x} \mathrm{d}x = \int_0^1 \frac{f(x)}{x} \mathrm{d}x + \sum_{k=1}^{\lfloor t \rfloor -1} \int_k^{k+1} \frac{f(\{x\})}{x} \mathrm{d}x + \int_{\lfloor t\rfloor }^t \frac{f(\{x\})}{x} \mathrm{d}x \]
and to notice that 
\[\sum_{k=1}^{\lfloor t \rfloor -1} \frac{1}{k+1} \int_0^1  f(x) \mathrm{d}x  \leq \sum_{k=1}^{\lfloor t \rfloor -1} \int_k^{k+1} \frac{f(\{x\})}{x} \mathrm{d}x \leq \sum_{k=1}^{\lfloor t \rfloor -1} \frac{1}{k} \int_0^1  f(x) \mathrm{d}x \]
and 
\[ \int_{\lfloor t\rfloor }^t \frac{f(\{x\})}{x} \mathrm{d}x \leq \frac{1}{\lfloor t \rfloor} \int_0^1  f(x) \mathrm{d}x.\]
\end{proof}

We are ready to prove point $(ii)$ of the theorem for the case $\nu = \frac{r}{s}$ with $\frac{r}{s} >1$ and $\gcd(r,s)=1$.
Let $a \in \R$. We want to show
\[\frac{X \left( a, \frac{r}{s}a \right) - \left( \frac{r}{s} -1 \right)a}{\sqrt{\theta \log a \left( \frac{1}{6} - \frac{1}{6sr} \right)}} \underset{a\to +\infty}{\overset{d}{\longrightarrow}} \mathcal{N}(0,1). \]

With the notation $T =\sum\limits_{y\in \mathcal{W}} \{(a+b)y\} - \{ay\}$, we established that as soon as $\V (T) \to \infty$, 
\[\frac{X (a,a+b) -b + \E (T)}{\sqrt{\V (T)}} \overset{d}{\longrightarrow} \mathcal{N}(0,1).\]
Set $b=\left( \frac{r}{s} -1 \right)a$. Using twice Lemma~\ref{lem:transfo} with the identity function and $t= \frac{r}{s}a$ and then $t = a$, we get by subtraction 
\[\E (T) = \theta \left( \frac{1}{2}\log \left( \frac{r}{s}a \right) - \frac{1}{2} \log a  + \mathcal{O}(1) \right) = \mathcal{O}_\theta (1). \]
Now, denoting $t:=\frac{a}{s}$, for all $x\in [0,1]$,
\[ \{ax \} - \left\{ \frac{r}{s} a x \right\} = \{stx \} - \{rtx  \} = (s-r) \{tx\} - \sum_{m=1}^{s-1} \mathds{1}_{\{tx \} \geq \frac{m}{s} } +  \sum_{n=1}^{r-1} \mathds{1}_{\{tx \} \geq \frac{n}{r} }.\]
Hence applying Lemma~\ref{lem:transfo} with the function $f: x\mapsto \left((s-r) x - \sum\limits_{m=1}^{s-1} \mathds{1}_{x \geq \frac{m}{s} } +  \sum_{n=1}^{r-1} \mathds{1}_{x \geq \frac{n}{r} }\right)^2$, we get
\[\int_0^1 \frac{\left( \{ax\} - \left\{ \frac{r}{s}ax \right\}\right)^2 }{x} \mathrm{d}x = \int_0^1 \frac{f(\{tx\})}{x} \mathrm{d}x \underset{t\to +\infty }{=} \log t \int_0^1 f(x) \mathrm{d}x + \mathcal{O}(1).\]
The author in \cite[Appendix B]{Bahier2017} shows that 
\[\int_0^1 f(x)\mathrm{d}x = \lim_{n \to \infty} \frac{1}{n} \sum_{j=1}^n (\{s j \alpha \} - \{r j \alpha \})^2  \]
where $\alpha$ is any arbitrary irrational number, and computes this limit explicitly, equal to $\frac{1}{6} - \frac{1}{6sr}$, which gives the claim. 
\subsubsection{Case $\nu$ irrational}

Let $\nu$ be an irrational number. For all $a>0$, let $\mu_a$ be the empirical measure of $(U_a, \nu U_a )$ on $(\R /\Z)^2$, where $U_a$ is a uniform random variable on $[0,a]$. 

Then, the Fourier transform of $\mu_a$ is given for all $(k,l)\in \Z^2$ by 
\[\widehat{\mu_a} (k,l) = \frac{1}{a} \int_0^a \mathrm{e}^{2i\pi (k + l\nu ) x} \mathrm{d}x \underset{a\to\infty}{\longrightarrow } \mathds{1}_{k+l\nu =0}.\]
Since $\nu$ is irrational, then $k+l\nu =0$ if and only if $(k,l)=(0,0)$. We deduce that $\mu_a$ converges to the Lebesgue measure of dimension $2$ on $[0,1]^2$. 

Let $f$ be a function from $(\R /\Z)^2$ to $\R$ defined by $f(x,y)=(x-y)^2$. $f$ is continuous everywhere, excepted on $\R/\Z \times \{\bar{0}\}$ and $\{\bar{0}\} \times \R/\Z$, which are of measure zero with respect to the Lebesgue measure of dimension $2$. Hence by the continuous mapping theorem, 
\[\int f \mathrm{d}\mu_a  \underset{a\to\infty}{\longrightarrow } \int_0^1 \int_0^1 (x-y)^2 \mathrm{d}x \mathrm{d}y = \frac{1}{6},\]
so that by a change of variables we get
\begin{equation}\label{eq:convirrat}
\int_0^1 (\{ ax\} - \{\nu a x\})^2 \mathrm{d}x = \frac{1}{a} \int_0^a (\{ x\} - \{\nu x\})^2  \mathrm{d} x = \int f \mathrm{d}\mu_a  \underset{a\to\infty}{\longrightarrow }  \frac{1}{6}. 
\end{equation}

It remains to show that this implies
\begin{equation}\label{eq:convlog}
\frac{1}{\log a} \int_0^1 \frac{(\{ ax\} - \{\nu a x\})^2}{x} \mathrm{d}x \underset{a\to\infty}{\longrightarrow }  \frac{1}{6}.
\end{equation}
Assume $a>1$. We write
\[\int_0^a \frac{f ( \{x\} , \{\nu x\})}{x} \mathrm{d}x = \int_0^1 \frac{f ( \{x\} , \{\nu x\})}{x} \mathrm{d}x  + \sum_{k=1}^{\lfloor a \rfloor -1} \int_k^{k+1} \frac{f ( \{x\} , \{\nu x\})}{x} \mathrm{d}x  + \int_{\lfloor a \rfloor }^a \frac{f ( \{x\} , \{\nu x\})}{x} \mathrm{d}x  \]
with $\int_0^1 \frac{f ( \{x\} , \{\nu x\})}{x} \mathrm{d}x < +\infty$, and $\int_{\lfloor a \rfloor }^a \frac{f ( \{x\} , \{\nu x\})}{x} \mathrm{d}x = \mathcal{O}(1)$ since $f$ is bounded. Moreover, for all integers $k\geq 1$,
\[ \frac{1}{k+1} \int_k^{k+1} f ( \{x\} , \{\nu x\}) \mathrm{d}x \leq \int_k^{k+1} \frac{f ( \{x\} , \{\nu x\})}{x} \mathrm{d}x \leq \frac{1}{k} \int_k^{k+1} f ( \{x\} , \{\nu x\}) \mathrm{d}x.\]
For the right-hand side inequality, denoting $a_k := \frac{1}{k}$, and $b_k:= \int_k^{k+1} f ( \{x\} , \{\nu x\}) \mathrm{d}x$, a summation by parts gives, for all $n \geq 1$, 
\begin{align*}
\sum_{k=1}^n a_k b_k &= a_n \sum_{k=1}^n b_k - \sum_{k=1}^{n-1} \sum_{m=1}^k b_m(a_{k+1}-a_k) \\
	&= \frac{1}{n} \int_1^{n+1} f ( \{x\} , \{\nu x\}) \mathrm{d}x + \sum_{k=1}^{n-1} \frac{1}{k(k+1)} \int_1^{k+1} f ( \{x\} , \{\nu x\}) \mathrm{d}x
\end{align*}
and from \eqref{eq:convirrat} we deduce
\begin{align*}
\sum_{k=1}^n a_k b_k &= \frac{1}{6} + o (1) + \sum_{k=1}^{n-1} \frac{1}{k} \left(\frac{1}{6} + o(1) \right) \\
	&= \frac{1}{6} \log n + o(\log n)
\end{align*}
as $n \to \infty$. Replacing $a_k$ by $\frac{1}{k+1}$ leads to the same asymptotic expression. Hence, from the squeeze theorem,
\[\frac{1}{\log a} \sum_{k=1}^{\lfloor a \rfloor -1} \int_k^{k+1} \frac{f ( \{x\} , \{\nu x\})}{x} \mathrm{d}x  \underset{a\to\infty}{\longrightarrow }  \frac{1}{6}, \] 
which gives \eqref{eq:convlog}.

\section{Translation of the limiting point process related to permutation matrices. Proof of Proposition~\ref{prop:transla}}
\label{sec:transla}

In this section we show that the translation of the limiting point process related to permutation matrices converges to the limiting point process related to modified permutation matrices. The precise statement corresponds to Proposition~\ref{prop:transla}. We will need the following lemma:

\begin{lem}
For all $j\in \N^*$,
\[( \{A y_1 \} , \{A y_2 \} , \dots , \{A y_j \} , y_1 , \dots , y_j ) \underset{A \to \infty}{\overset{d}{\longrightarrow}}  (\Phi_1 , \dots , \Phi_j , y_1 , \dots , y_j)\]
where $\Phi_1, \dots , \Phi_j$ are i.i.d random variables uniformly distributed on $[0,1)$ and independent of $y_1, \dots , y_j$. 
\end{lem}
\begin{proof}
Let $j \in \N^*$. We know that $\vec{y}:=(y_1 , \dots , y_j)$ has a density with respect to the Lebesgue measure (see \cite{arratia2003logarithmic}). Hence, for all $\vec{k} \in \Z^j$ and $\vec{\lambda} \in \R^j$, 
\[\E \left[ \mathrm{e}^{2i\pi A \vec{k} \cdot \vec{y} + i \vec{\lambda} \cdot \vec{y} }  \right] = \E \left[ \mathrm{e}^{i\vec{y} \cdot (2\pi A \vec{k} + \vec{\lambda} )}  \right] = \widehat{\mu_{\vec{y}}}  (2\pi A \vec{k} + \vec{\lambda}) \underset{A\to \infty}{\longrightarrow} 0 \]
as soon as $\vec{k}\neq 0$, applying Riemann-Lebesgue lemma.
\end{proof}

We are ready to prove Proposition~\ref{prop:transla}.

Let $f$ be a continuous function from $\R$ to $\C$ such that $\mathrm{supp} f \subset [-M , M]$ for any $M>0$. With the same notations as in the previous lemma, we want to show:
\begin{equation}\label{eq:transla}
\sum_{\substack{ k\in \Z \setminus \{0\} \\ j\geq 1  }} f \left( \frac{k}{y_j} - A \right)  \underset{A \to \infty}{\overset{d}{\longrightarrow}}  \sum_{\substack{ k\in \Z \\ j\geq 1  }} f \left( \frac{k - \Phi_j}{y_j} \right).
\end{equation}
Let $j_0 \in \N^*$. The probability that there exists non-zero terms in the sum $\sum\limits_{\substack{ k\in \Z \\ j\geq j_0  }}  f \left( \frac{k-\Phi_j}{y_j} \right)$ is 
\begin{align*}
\PP \left( \exists k \in \Z , \ \exists j \geq j_0 , \left\vert \frac{k-\Phi_j}{y_j} \right\vert \leq M \right) &\leq  \sum_{j\geq j_0} \sum_{k\in \Z} \PP \left( \left\vert \frac{k-\Phi_j}{y_j} \right\vert \leq M \right) \\
	&=  \sum_{j\geq j_0}  \left[ \PP (\Phi_j \leq M y_j) + \PP (1-\Phi_j \leq M y_j) + \sum_{k = -\lfloor M \rfloor -1}^{-1} \PP \left( y_j \geq \frac{\vert k - \Phi_j \vert}{M} \right)  \right. \\
	&\qquad + \left. \sum_{k = 2}^{\lfloor M \rfloor +1} \PP \left( y_j \geq \frac{\vert k - \Phi_j \vert}{M} \right) \right] \\
	&\leq \sum_{j\geq j_0} \left[ M r^j + M r^j + \sum_{k = -\lfloor M \rfloor -1}^{-1} \frac{M r^j}{-k} + \sum_{k = 2}^{\lfloor M \rfloor +1} \frac{M r^j}{k-1}   \right] \\
	&\leq \frac{2M(1+ \log(M+1))}{1-r} r^{j_0} 
\end{align*}
where we recall that $r$ is the constant given by \eqref{eq:yjrj}.
Thus 
\begin{equation}\label{eq:Yj}
\PP \left( \exists j\geq j_0,\ \exists k \in \Z,\  f \left( \frac{k-\Phi_j}{y_j} \right) \neq 0 \right) \underset{j_0 \to\infty}{\longrightarrow} 0.
\end{equation}

Let $A>2M$. The probability that there exists non-zero terms in the sum $\sum\limits_{\substack{ k\in \Z \setminus \{0\} \\ j\geq j_0  }}  f \left( \frac{k}{y_j} -A \right)$ is 
\[\PP \left( \exists k \in \Z\setminus \{0\} , \ \exists j \geq j_0 , \ \left\vert \frac{k}{y_j} - A \right\vert \leq M \right) = \PP \left( \exists k \in \Z \cap [1, M+A] , \ \exists j \geq j_0 , \ \frac{k}{M+A} \leq y_j \leq \frac{k}{A-M} \right) \]
since $0<y_j<1$ a.s. for all $j$. Moreover, the intervals $\left[ \frac{k}{M+A} , \frac{k}{A-M}  \right]$ and $\left[ \frac{k+1}{M+A} , \frac{k+1}{A-M}  \right]$ overlap if and only if $k \geq \frac{A-M}{2M}$. \\
On the one hand,
\begin{align*}
\PP \left( \exists k \in \Z \cap \left[ \frac{A-M}{2M} , M+A \right] , \ \exists j \geq j_0 , \ \frac{k}{M+A} \leq y_j \leq \frac{k}{A-M} \right) &\leq \PP \left(  \exists j\geq j_0 , \ y_j \geq \frac{\frac{A-M}{2M}}{M+A} \right) \\
	&\leq \sum_{j\geq j_0} \PP \left( y_j \geq \frac{1}{6M} \right) \\
	&\leq 6M \sum_{j\geq j_0} r^j \\
	&\underset{j_0 \to +\infty}{\longrightarrow} 0.
\end{align*}
On the other hand, assuming $j_0\geq 3$, it is easy to check that $(y_j)_{j\geq j_0} \overset{d}{=} (Py_j)_{j\geq 2}$ where $P:=U_2 \dots U_{j_0 -1}$ is a product of $j_0-2$ independent Beta$(\theta , 1)$ random variables, and we deduce
\begin{align*}
&\PP \left( \exists k \in \Z \cap \left[ 1, \frac{A-M}{2M} \right] , \ \exists j \geq j_0 , \ \frac{k}{M+A} \leq y_j \leq \frac{k}{A-M} \right) \\
&\qquad = \PP \left( \exists k \in \Z \cap \left[ 1, \min \left( \frac{A-M}{2M} , P(M+A) \right) \right] , \ \exists j \geq 2 , \ \frac{k}{P(M+A)} \leq y_j \leq \frac{k}{P(A-M)} \right).
\end{align*} 
Conditionally to $P$, the corresponding quantity is bounded by the probability that there is at least one point of $\mathcal{X}$ located in the disjoint union 
\[\bigcup\limits_{k \in \Z \cap \left[ 1, \min \left( \frac{A-M}{2M} , P(M+A) \right) \right]}  \left[ \frac{k}{P(M+A)} , \frac{k}{P(A-M)} \right],  \]
hence
\begin{align*}
&\PP \left( \exists k \in \Z \cap \left[ 1, \frac{A-M}{2M} \right] , \ \exists j \geq j_0 , \ \frac{k}{M+A} \leq y_j \leq \frac{k}{A-M} \right)  \\
	&\qquad \leq \E \left( 1- \exp \left( - \sum_{1\leq k \leq \min \left( \frac{A-M}{2M} , P(M+A) \right)}   \int_{\frac{k}{P(M+A)}}^{\frac{k}{P(A-M)}} \frac{\theta}{x} \mathrm{d}x \right) \right) \\
	&\qquad \leq 1 - \E \left( \exp \left( -\theta \min \left( \frac{A-M}{2M} , P(M+A) \right) \log \left( \frac{A+M}{A-M}  \right) \right) \right) \\
	&\qquad \leq  1 - \E \left( \exp \left( -\theta P \frac{2M(M+A)}{A-M} \right) \right)  \\
	&\qquad \leq 1 - \E \left( \exp \left( -6M \theta P  \right) \right) .
\end{align*}
In addition, $P$ converges almost surely to $0$ when $j_0$ goes to $+\infty$, and $0 \leq \exp (-6M\theta P) \leq 1$, then by dominated convergence $\E \left( \exp \left( -6M \theta P  \right) \right) \underset{j_0\to +\infty}{\longrightarrow} 1$. \\
Consequently, 
\begin{equation}\label{eq:YjA}
\sup\limits_A \PP \left( \exists j\geq j_0,\ \exists k \in \Z \setminus \{0\},\  f \left( \frac{k}{y_j} -A \right) \neq 0 \right) \underset{j_0 \to\infty}{\longrightarrow} 0.
\end{equation}

Furthermore, it is easy to check that
\[\sum_{j=1}^{j_0} \sum_{k\in \Z} f \left( \frac{k-\Phi_j}{y_j} \right) = \sum_{(j,k)\in S_{j_0}} f \left( \frac{k-\Phi_j}{y_j} \right) \]
and
\[\sum_{j=1}^{j_0} \sum_{k\in \Z \setminus \{0\}} f \left( \frac{k}{y_j} - A \right) = \sum_{(j,k)\in S_{j_0}} f \left( \frac{k-\{A y_j \}}{y_j} \right)\]
where $S_{j_0}:=\{ (j,k) : \ 1\leq j \leq j_0 , \ \vert k \vert \leq M+1 \}$ is finite. Using the previous lemma and the continuous mapping theorem we deduce
\begin{equation}\label{eq:translafinite}
\sum_{(j,k)\in S_{j_0}} f \left( \frac{k-\{A y_j \}}{y_j} \right)  \underset{A \to \infty}{\overset{d}{\longrightarrow}} \sum_{(j,k)\in S_{j_0}} f \left( \frac{k-\Phi_j}{y_j} \right). 
\end{equation}

Now, let $g$ be a continuous and bounded function from $\R$ to $\R$. For all $j$ and $A$, let $V_{j,A}:= \sum\limits_{k\in \Z \setminus \{0\}} f \left( \frac{k}{y_j} - A \right)$, and $V_j:= \sum\limits_{k\in \Z} f \left( \frac{k-\Phi_j}{y_j} \right)$. For all $j_0\geq 1$, denoting $\Omega_{j_0} := \{\forall j > j_0 , \ V_{j,A} =0 \}$, we have
\begin{align*}
\E \left( g \left(  \sum_{j\geq 1} V_{j,A} \right) \right) &= \E \left( g \left(  \sum_{j\leq j_0} V_{j,A} \right) \mathds{1}_{\Omega_{j_0}} \right) + \E \left( g \left(  \sum_{j\geq 1} V_{j,A} \right)  \mathds{1}_{\Omega_{j_0}^\complement}  \right) \\
	&= \E \left( g \left(  \sum_{j\leq j_0} V_{j,A} \right) \right) -\E \left( g \left(  \sum_{j\leq j_0} V_{j,A} \right) \mathds{1}_{\Omega_{j_0}^\complement} \right) + \E \left( g \left(  \sum_{j\geq 1} V_{j,A} \right)  \mathds{1}_{\Omega_{j_0}^\complement}  \right) \\
	&= \E \left( g \left(  \sum_{j\leq j_0} V_{j,A} \right) \right) + \mathcal{O}( \PP ( \Omega_{j_0}^\complement )).
\end{align*}
Hence
\begin{align*}
\underset{A \to \infty}{\overline{\underline{ \lim }}} \  \E \left( g \left(  \sum_{j\geq 1} V_{j,A} \right) \right) &= \underset{A \to \infty}{\overline{\underline{ \lim }}} \ \E \left( g \left(  \sum_{j\leq j_0} V_{j,A} \right) \right) + \mathcal{O}( \sup\limits_{A} \PP ( \exists j \geq j_0 , \ V_{j,A} \neq 0 )) \\
	&= \E \left( g \left(  \sum_{j\leq j_0} V_j \right) \right) +  \mathcal{O}( \sup\limits_{A} \PP ( \exists j \geq j_0 , \ V_{j,A} \neq 0 )) \\
	&= \E \left( g \left(  \sum_{j\geq 1} V_j \right) \right) + \mathcal{O}(\PP ( \exists j \geq j_0 , \ V_j \neq 0 ) + \sup\limits_{A} \PP ( \exists j \geq j_0 , \ V_{j,A} \neq 0 )) \\
	&\underset{j_0 \to \infty}{=} \E \left( g \left(  \sum_{j\geq 1} V_j \right) \right) + o(1)
\end{align*}
where the second equality derives from the convergence in distrbution \eqref{eq:translafinite}, and the last equality follows from \eqref{eq:Yj} and \eqref{eq:YjA}. This gives \eqref{eq:transla}.

\begin{remarque}
This result can be easily extended to simple functions (linear combination of indicator functions) with compact support. In particular we have the following corollary: 
\end{remarque}

\begin{corollaire}\label{cor:transla}
Let $s,t \in \R_+$. Using the notations of Sections~\ref{sec:limitMPM} and \ref{sec:limitMP}, we have
\[X(s,s+t) \underset{s \to +\infty}{\overset{d}{\longrightarrow}} \widetilde{X} (t).\]
\end{corollaire}
\begin{proof}
It suffices to write
\[X(s,s+t) = \tau_\infty ((s,s+t]) = \sum_{\substack{ k\in \Z \setminus \{0\} \\ j\geq 1  }} \mathds{1}_{ \frac{k}{y_j} - s  \in (0,t]},\]
\[\widetilde{X}(t) = \widetilde{\tau}_\infty ((0,t]) = \sum_{\substack{ k\in \Z \\ j\geq 1 }} \mathds{1}_{  \frac{k - \Phi_j}{y_j} \in (0,t]},\]
and for all $x \in \R$, a.s., $\sum\limits_{\substack{ k\in \Z \\ j\geq 1  }} \mathds{1}_{ \frac{k - \Phi_j}{y_j} = x} =0$, so the continuous mapping theorem applies with $f=\mathds{1}_{(0,t]}$ under a similar reasoning as in the previous proof. 
\end{proof}

\section*{Appendix}

In this section we prove Lemmas~\ref{lem:calcul}, \ref{lem:calcul2} and \ref{lem:calcul3}.

\subsection*{Proof of Lemma~\ref{lem:calcul}}

Let $n\in \N^*$. A simple change of variables ($t=nx$) gives
\[\int_0^1 \frac{\{nx\}}{x} \mathrm{d}x = \sum_{k=0}^{n-1} \int_0^1 \frac{t}{t+k} \mathrm{d}t = 1 + \sum_{k=1}^{n-1} \int_0^1 \frac{t}{t+k} \mathrm{d}t,\]
with for all $k\geq 1$ and $t\in [0,1]$,  $\frac{1}{k+1} \leq \frac{1}{t+k} \leq \frac{1}{k}$, thus 
\[1 + \frac{1}{2} \sum_{k=1}^{n-1} \frac{1}{k+1} \leq \int_0^1 \frac{\{nx\}}{x} \mathrm{d}x \leq 1 + \frac{1}{2} \sum_{k=1}^{n-1} \frac{1}{k} \] so that 
\[\int_0^1 \frac{\{nx\}}{x} \mathrm{d}x \underset{n\to \infty}{=} \frac{1}{2} \log n + \mathcal{O}(1).\]
The same change of variables leads to
\begin{align*}
\int_0^1 \{nx\} \log x \mathrm{d}x &= \sum_{k=0}^{n-1} \int_0^1 \frac{t}{n} \log \left( \frac{t+k}{n} \right) \mathrm{d}t \\
	&= -\frac{1}{2} \log n + \frac{1}{n} \sum_{k=0}^{n-1} \int_0^1 t \log (t+k) \mathrm{d}t \\
	&= -\frac{1}{2} \log n + \frac{1}{n} \sum_{k=0}^{n-1} \left( \frac{1}{2} \log (1+k) - \frac{1}{2} \int_0^1  \frac{t^2}{t+k} \mathrm{d}t \right) \\
	&= -\frac{1}{2} \log n + \frac{1}{2n} \log (n!) - \frac{1}{4n} - \frac{1}{2n} \sum_{k=1}^{n-1} \left( \frac{1}{2} - k + k^2 \log \left( 1+ \frac{1}{k} \right) \right).
\end{align*}
Moreover, as a consequence of Stirling's formula, 
\[\log (n!) \underset{n\to \infty}{=} n \log n - n + \frac{1}{2} \log n + \mathcal{O}(1),\]
and furthermore we have
\[\frac{1}{2} - k + k^2 \log \left( 1+ \frac{1}{k} \right) \underset{k\to \infty}{=} \frac{1}{3k} + \mathcal{O} \left( \frac{1}{k^2} \right).\]
We deduce
\begin{align*}
\int_0^1 \{nx\} \log x \mathrm{d}x &\underset{n\to \infty}{=} -\frac{1}{2} + \frac{1}{4n} \log n + \mathcal{O} \left( \frac{1}{n} \right) - \frac{1}{6n} \log n + \mathcal{O} \left( \frac{1}{n} \right) \\
	&= -\frac{1}{2} + \frac{1}{12n} \log n + \mathcal{O} \left( \frac{1}{n} \right).
\end{align*}

\subsection*{Proof of Lemma~\ref{lem:calcul2}}

Let $\ell,n \in \N^*$. A summation by parts gives
\begin{align*}
\sum_{k=1}^{n} \frac{k}{n} \log \frac{k}{n} &= - \sum_{k=1}^{n-1} \log \left( 1 + \frac{1}{k} \right) \sum_{j=1}^k \frac{j}{n} \\
	&= -\sum_{k=1}^{n-1}  \left( \frac{1}{k} + \mathcal{O} \left( \frac{1}{k^2} \right) \right) \frac{k(k+1)}{2n} \\
	&= -\sum_{k=1}^{n-1} \left(\frac{k}{2n} +\mathcal{O} \left( \frac{1}{n}\right) \right) \\
	&= -\frac{n(n+1)}{4n} + \mathcal{O}(1)   \\
	&= -\frac{n}{4} + \mathcal{O}(1).
\end{align*}
Besides, for all fixed $t \in (0,1)$, 
\begin{align*}
\sum_{k=1}^n \mathds{1}_{\frac{k}{n} \geq t} \log \frac{k}{n}  &= - \sum_{k=\lceil nt \rceil}^{n-1} \log \left( 1 + \frac{1}{k} \right) \sum_{j=\lceil nt \rceil}^k 1 \\
	&= - \sum_{k=\lceil nt \rceil}^{n-1}  \left( \frac{1}{k} + \mathcal{O} \left( \frac{1}{k^2} \right) \right) (k-\lceil nt \rceil +1) \\
	&= - \sum_{k=\lceil nt \rceil}^{n-1}  \left( 1 - \frac{\lceil nt \rceil}{k} + \mathcal{O} \left( \frac{1}{k} \right) \right) \\
	&\underset{n\to \infty}{=} -n (1-t + t \log t) + \mathcal{O}(1).
\end{align*}
Thus, on the one hand, 
\begin{align*}
\sum_{k=1}^{n-1} \left\{ \ell \frac{k}{n} \right\} \log \frac{k}{n} &= \ell \left(\sum_{k=1}^{n-1} \frac{k}{n} \log \frac{k}{n}\right) - \sum_{m=1}^{\ell-1} \sum_{k=1}^{n-1} \mathds{1}_{ \frac{k}{n} \geq \frac{m}{\ell}} \log \frac{k}{n} \\
	&= \ell \left( -\frac{n}{4} + \mathcal{O}(1) \right) - \sum_{m=1}^{\ell-1} \left( -n \left(1- \frac{m}{\ell} + \frac{m}{\ell} \log \frac{m}{\ell} \right) + \mathcal{O}(1) \right) \\
	&= \left[  -\frac{\ell}{4} + \sum_{m=1}^{\ell-1} \left(1- \frac{m}{\ell} + \frac{m}{\ell} \log \frac{m}{\ell} \right) \right] n + \mathcal{O} (1) \\
	&= \left[  \frac{\ell}{4} - \frac{1}{2} + \sum_{m=1}^{\ell-1}  \frac{m}{\ell} \log \frac{m}{\ell} \right] n + \mathcal{O} (1),
\end{align*}
and on the other hand, 
\begin{align*}
\sum_{k=1}^{n-1} \left\{-\ell \frac{k}{n} \right\} \log \frac{k}{n} &= -\ell \left(\sum_{k=1}^{n-1} \frac{k}{n} \log \frac{k}{n}\right) + \sum_{m=0}^{\ell-1} \sum_{k=1}^{n-1} \mathds{1}_{ \frac{k}{n} > \frac{m}{\ell}} \log \frac{k}{n} \\
	&= \left[  \frac{\ell}{4} - \sum_{m=1}^{\ell-1} \left(1- \frac{m}{\ell} + \frac{m}{\ell} \log \frac{m}{\ell} \right) \right] n +  \mathcal{O} (1) + \sum_{k=1}^{n-1} \mathds{1}_{\frac{k}{n}>0} \log \frac{k}{n} \\
	&=  \left[ -\frac{\ell}{4} + \frac{1}{2} - \sum_{m=1}^{\ell-1} \frac{m}{\ell} \log \frac{m}{\ell} \right] n +  \mathcal{O} (1) + \sum_{k=1}^{n-1} \log \frac{k}{n} .
\end{align*}
Finally, it just remains to see
\[\sum_{k=1}^n \log \frac{k}{n} = - \sum_{k=1}^{n-1} k \log \left( 1+ \frac{1}{k} \right) = -n + \frac{1}{2} \log n + \mathcal{O}(1). \]

\subsection*{Proof of Lemma~\ref{lem:calcul3}}

Let $f:x \mapsto (\{ px \} - \{qx\})^2$.
Denote for all positive integers $m$, $E_m = \left\{ \frac{k}{m} ; \ 1\leq k \leq m-1  \right\}$, and let $E_{p,q} = E_p \cup E_q$. 
Noticing that $x\mapsto (\{px\} - \{qx\})$ is a piecewise linear function with constant slope equal to $p-q$ and jumps at multiples of $1/p$ and multiples of $1/q$, the derivative of the distribution $T_f$ related to $f$ on $(0,1)$ is given by
\begin{align*}
(T_f)^\prime &= T_{f^\prime} + \sum_{r \in E_{p,q}} (f(r+0) - f(r-0)) \delta_r \\
	&= 2(p-q) (\{p \cdot \} - \{q \cdot \}) + \sum_{r \in E_{p,q}} (f(r+0) - f(r-0)) \delta_r 
\end{align*}
Thus, integrating by parts gives
\[\int_0^1 \frac{f(x)}{x} \mathrm{d}x =-2 (p-q) \int_0^1 (\{px \} - \{qx \})\log x \mathrm{d}x - \sum_{r\in E_{p,q}} (f(r+0) - f(r-0))\log r.\]
If $r \in E_p \cap E_q$, it is easy to check that $f(r+0) - f(r-0)=0$. \\
If $r= \frac{k}{p} \not\in E_q$ then $f(r+0) - f(r-0) = \left\{ q \frac{k}{p} \right\}^2 - \left( 1 - \left\{ q \frac{k}{p} \right\} \right)^2 = 2 \left\{ q \frac{k}{p} \right\} -1.$ \\
Symmetrically if $r= \frac{j}{q} \not\in E_p$ then $f(r+0) - f(r-0) =  2 \left\{ p \frac{j}{q} \right\} -1.$ \\
Finally we get
\begin{align*}
\sum_{r \in E_{p,q}} (f(r+0) - f(r-0)) \log r &= \sum_{\substack{1\leq k \leq p-1 \\ \frac{1}{q} \nmid \frac{k}{p} }} \left(2 \left\{ q \frac{k}{p} \right\} -1 \right) \log \left( \frac{k}{p} \right) + \sum_{\substack{1\leq j \leq q-1 \\ \frac{1}{p} \nmid \frac{j}{q} }} \left(2 \left\{ p \frac{j}{q} \right\} -1 \right) \log \left( \frac{j}{q} \right) \\
	&=\sum_{k=1}^{p-1} \left( 2\left\{ q\frac{k}{p} \right\} - 1 \right) \log \left( \frac{k}{p} \right) + \sum_{j=1}^{q-1} \left( 2\left\{ p\frac{j}{q} \right\} - 1 \right) \log \left( \frac{j}{q} \right)  \\
	&\qquad + 2 \sum_{m=1}^{\gcd (p,q) -1} \log \left(\frac{m}{\gcd (p,q)}  \right).
\end{align*}

\vspace*{1cm}

\textbf{Acknowledgements}: The author wishes to acknowledge the help provided by his PhD advisor Joseph Najnudel.

\bibliographystyle{plain}
\bibliography{biblio}

\begin{thebibliography}{10}

\bibitem{arratia1998central}
Richard Arratia.
\newblock On the central role of scale invariant {P}oisson processes on (0,
  oo).
\newblock {\em Microsurveys in discrete probability (Princeton, NJ, 1997)},
  pages 21--41, 1998.

\bibitem{arratia2003logarithmic}
Richard Arratia, Andrew~D Barbour, and Simon Tavar{\'e}.
\newblock {\em Logarithmic combinatorial structures: a probabilistic approach},
  volume~1.
\newblock European Mathematical Society Z{\"u}rich, 2003.

\bibitem{arratia2006tale}
Richard Arratia, Andrew~D Barbour, and Simon Tavar{\'e}.
\newblock A tale of three couplings: Poisson--{D}irichlet and {GEM}
  approximations for random permutations.
\newblock {\em Combinatorics, Probability and Computing}, 15(1-2):31--62, 2006.

\bibitem{Bahier2017}
Valentin Bahier.
\newblock On the number of eigenvalues of modified permutation matrices in
  mesoscopic intervals.
\newblock {\em Journal of Theoretical Probability}, Dec 2017.

\bibitem{ben2015fluctuations}
G{\'e}rard Ben~Arous and Kim Dang.
\newblock On fluctuations of eigenvalues of random permutation matrices.
\newblock {\em Annales de L'Institut Henri Poincare Section (B) Probability and
  Statistics}, 51:620--647, 2015.

\bibitem{dang2014characteristic}
Kim Dang and Dirk Zeindler.
\newblock The characteristic polynomial of a random permutation matrix at
  different points.
\newblock {\em Stochastic Processes and their Applications}, 124(1):411--439,
  2014.

\bibitem{kerov1997stick}
Sergei Kerov and Natalia Tsilevich.
\newblock Stick breaking process generated by virtual permutations with ewens
  distribution.
\newblock {\em Journal of Mathematical Sciences}, 87(6):4082--4093, 1997.

\bibitem{kingman1993poisson}
John Frank~Charles Kingman.
\newblock {\em Poisson processes}.
\newblock Wiley Online Library, 1993.

\bibitem{najnudel2013distribution}
Joseph Najnudel and Ashkan Nikeghbali.
\newblock The distribution of eigenvalues of randomized permutation matrices
  [sur la distribution des valeurs propres de matrices de permutation
  randomis{\'e}es].
\newblock In {\em Annales de l'institut Fourier}, volume~63, pages 773--838,
  2013.

\bibitem{patil1977diversity}
Ganapati Patil and C~Taillie.
\newblock Diversity as a concept and its implications for random communities.
\newblock {\em Bull. Int. Stat. Inst}, 47:497--515, 1977.

\bibitem{tsilevich1997distribution}
Natalia Tsilevich.
\newblock Distribution of cycle lengths of infinite permutations.
\newblock {\em Journal of Mathematical Sciences}, 87(6):4072--4081, 1997.

\bibitem{wieand2003permutation}
Kelly Wieand.
\newblock Permutation matrices, wreath products, and the distribution of
  eigenvalues.
\newblock {\em Journal of Theoretical Probability}, 16(3):599--623, 2003.

\bibitem{wieand2000eigenvalue}
Kelly Wieand et~al.
\newblock Eigenvalue distributions of random permutation matrices.
\newblock {\em The Annals of Probability}, 28(4):1563--1587, 2000.

\end{thebibliography}

\end{document}